\documentclass[11pt]{amsart}
\usepackage{amssymb}
\usepackage{verbatim}
\usepackage{hyperref}
\usepackage{color}
\begin{document}
\setcounter{tocdepth}{1}

% define theorem environments
\newtheorem{theorem}{Theorem}    %[section]
\newtheorem{proposition}[theorem]{Proposition}
\newtheorem{conjecture}[theorem]{Conjecture}
\def\theconjecture{\unskip}
\newtheorem{corollary}[theorem]{Corollary}
\newtheorem{lemma}[theorem]{Lemma}
\newtheorem{sublemma}[theorem]{Sublemma}
\newtheorem{fact}[theorem]{Fact}
\newtheorem{observation}[theorem]{Observation}
\theoremstyle{definition}
\newtheorem{definition}{Definition}
\newtheorem{notation}[definition]{Notation}
\newtheorem{remark}[definition]{Remark}
\newtheorem{question}[definition]{Question}
\newtheorem{questions}[definition]{Questions}

\newtheorem{example}[definition]{Example}
\newtheorem{problem}[definition]{Problem}
\newtheorem{exercise}[definition]{Exercise}

\numberwithin{theorem}{section}
\numberwithin{definition}{section}
\numberwithin{equation}{section}

\def\reals{{\mathbb R}}
\def\torus{{\mathbb T}}
\def\heis{{\mathbb H}}
\def\integers{{\mathbb Z}}
\def\rationals{{\mathbb Q}}
\def\naturals{{\mathbb N}}
\def\complex{{\mathbb C}\/}
\def\distance{\operatorname{distance}\,}
\def\sym{\operatorname{Symm}\,}
\def\support{\operatorname{support}\,}
\def\dist{\operatorname{dist}}
\def\Span{\operatorname{span}\,}
\def\degree{\operatorname{degree}\,}
\def\kernel{\operatorname{kernel}\,}
\def\dim{\operatorname{dim}\,}
\def\codim{\operatorname{codim}}
\def\trace{\operatorname{trace\,}}
\def\Span{\operatorname{span}\,}
\def\dimension{\operatorname{dimension}\,}
\def\codimension{\operatorname{codimension}\,}
\def\Gl{\operatorname{Gl}}
\def\nullspace{\scriptk}
\def\kernel{\operatorname{Ker}}
\def\Re{\operatorname{Re} }
\def\Im{\operatorname{Im} }
\def\eps{\varepsilon}
\def\lt{L^2}
\def\diver{\operatorname{div}}
\def\curl{\operatorname{curl}}
\newcommand{\norm}[1]{ \|  #1 \|}
\def\expect{\mathbb E}
\def\bull{$\bullet$\ }
\def\det{\operatorname{det}}
\def\Det{\operatorname{Det}}
\def\bestA{\mathbf A}
\def\bestC{\mathbf C}
\def\bestAqd{\mathbf A_{q,d}}
\def\bestBqd{\mathbf B_{q,d}}
\def\bestB{\mathbf B}
\def\bestC{\mathbf C}
\def\Apq{\mathbf A_{p,q}}
\def\Apqr{\mathbf A_{p,q,r}}
\def\rank{\operatorname{rank}}
\def\diameter{\operatorname{diameter}}
\def\essinf{\operatorname{ess\,inf}}
\def\esssup{\operatorname{ess\,sup}}

\newcommand{\abr}[1]{ \langle  #1 \rangle}
\def\unitQ{{\mathbf Q}}
\def\mbfp{{\mathbf P}}

\def\aff{\operatorname{Aff}}
\def\T{{\mathcal T}}

\def\bz{\mathbf z}
\def\bG{\mathbf G}
\def\fg{\mathfrak G}

\newcommand{\Norm}[1]{ \Big\|  #1 \Big\| }
\newcommand{\set}[1]{ \left\{ #1 \right\} }
\newcommand{\sset}[1]{ \{ #1 \} }

\def\one{{\mathbf 1}}
\def\bb{{{\mathbb B}}}
\def\onei{{\mathbf 1}_I}
\def\onee{{\mathbf 1}_E}
\def\onea{{\mathbf 1}_A}
\def\oneb{{\mathbf 1}_B}
\def\onebb{{\mathbf 1}_{\bb}}
\def\wonee{\widehat{\mathbf 1}_E}
\newcommand{\modulo}[2]{[#1]_{#2}}

\def\repair{\medskip\hrule\hrule\medskip}

\def\scriptf{{\mathcal F}}
\def\scripts{{\mathcal S}}
\def\scriptq{{\mathcal Q}}
\def\scriptg{{\mathcal G}}
\def\scriptm{{\mathcal M}}
\def\scriptb{{\mathcal B}}
\def\scriptc{{\mathcal C}}
\def\scriptt{{\mathcal T}}
\def\scripti{{\mathcal I}}
\def\scripte{{\mathcal E}}
\def\scriptv{{\mathcal V}}
\def\scriptw{{\mathcal W}}
\def\scriptu{{\mathcal U}}
\def\scripta{{\mathcal A}}
\def\scriptr{{\mathcal R}}
\def\scripto{{\mathcal O}}
\def\scripth{{\mathcal H}}
\def\scriptd{{\mathcal D}}
\def\scriptl{{\mathcal L}}
\def\scriptn{{\mathcal N}}
\def\scriptp{{\mathcal P}}
\def\scriptk{{\mathcal K}}
\def\scriptP{{\mathcal P}}
\def\scriptj{{\mathcal J}}
\def\scriptz{{\mathcal Z}}
\def\frakv{{\mathfrak V}}
\def\frakE{{\mathfrak E}}
\def\frakG{{\mathfrak G}}
\def\frakA{{\mathfrak A}}
\def\frakB{{\mathfrak B}}
\def\frakC{{\mathfrak C}}
\def\frakf{{\mathfrak F}}
\def\fraki{{\mathfrak I}}
\def\fcross{{\mathfrak F^{\times}}}

\def\symdif{\,\Delta\,}
\def\defe{\dist(E,\frakE)}
\def\defb{|E\symdif \bb|}

\def\ball{\mathbb B}
\def\bp{\mathbf p}
\def\bx{\mathbf x}
\def\by{\mathbf y}
\def\ba{\mathbf a}
\def\bg{\mathbf g}
\def\bt{\mathbf t}
\def\bs{\mathbf s}
\def\bff{\mathbf f}
\def\bF{\mathbf F}
\def\bz{\mathbf z}
\def\bG{\mathbf G}
\def\fg{\mathfrak G}
\def\symG{\scriptg}
\def\sp{\rm Sp}

\author{Michael Christ}
\address{
        Michael Christ\\
        Department of Mathematics\\
        University of California \\
        Berkeley, CA 94720-3840, USA}
\email{mchrist@berkeley.edu}

\date{January 24, 2016. Emendation June 3, 2017.}
% begun 10/22/2015 
\thanks{Research supported by National Science Foundation grant DMS-1363324.}

\title[Young's convolution inequality for Heisenberg groups]
{On Young's convolution inequality \\ for Heisenberg groups}

\begin{abstract} 
Young's convolution inequality provides an upper bound for the convolution of functions
in terms of $L^p$ norms. It is known that for certain groups, including Heisenberg groups, the optimal
constant in this inequality is equal to that for Euclidean space
of the same topological dimension, yet no extremizing functions exist.
For Heisenberg groups we characterize ordered triples of functions that nearly extremize
the inequality. 

The analysis relies on a characterization of approximate solutions of a certain class
of functional equations. A result of this type is developed for a class of such equations. 
% up to a factor close to $1$.
\end{abstract}

\maketitle
% \tableofcontents

\section{Introduction}

This paper characterizes ordered triples of functions that
nearly extremize Young's convolution inequality for Heisenberg groups.
We first review Young's inequality with sharp constant for Euclidean spaces,
then review the corresponding inequality for Heisenberg groups, 
recalling observations of Klein and Russo \cite{kleinrusso}
and of Beckner \cite{beckner1995}
concerning the distinction between the Euclidean and Heisenberg settings.
For Heisenberg groups we introduce a group of symmetries of the inequality,
along with a special class of ordered triples of Gaussian functions.
Our main theorem states that an ordered triple of functions
nearly extremizes the inequality if and only if it differs 
by a small amount, in the relevant norm, from
the image of one of these special ordered triples of Gaussians
under some element of the symmetry group.
Our conclusion is of ``$o(1)$'' type; we do not obtain an explicit
upper bound on the difference of norms as a function of the 
discrepancy from exact extremization.

The proof combines a preexisting characterization of near extremizers of
Young's inequality for Euclidean groups
with the structure of Heisenberg groups
and with a characterization of approximate solutions
of certain functional equations.

\subsection{Young's inequality for Euclidean groups}

In its classical form, Young's convolution inequality for the Euclidean group 
$\reals^m$ states that 
the convolution $f*g$ of functions $f,g$ satisfies the upper bound
\begin{equation}\label{classicyoung} \norm{f*g}_{L^r(\reals^m)}
\le \norm{f}_{L^p(\reals^m)}\norm{g}_{L^q(\reals^m)} \end{equation}
whenever $p,q,r\in[1,\infty]$ and $r^{-1} = p^{-1}+q^{-1}-1$.
In its sharp form established by Beckner \cite{beckner} for 
the case when all three of  $p,q,r'$ are less than or equal to $2$, 
and subsequently established independently by Brascamp and Lieb \cite{brascamplieb} 
and by Beckner for the full range of exponents, it states that 
\begin{equation} \norm{f*g}_{L^r(\reals^m)}\le 
\bestC_{p,q}^n \norm{f}_{L^p(\reals^m)}\norm{g}_{L^q(\reals^m)} \end{equation}
with
\begin{equation} \bestC_{p,q} = A_pA_q A_{r'} 
\ \text{ where }\  A_s = s^{1/2s} t^{-1/2t} \text{ with $t=s'$;} \end{equation}
here and below
$s'$ denotes the exponent $s' = s/(s-1)$ conjugate to $s$.
The factor $\bestC_{p,q}$ is strictly less than $1$ provided that $p,q,r\in(1,\infty)$,
and $\bestC_{p,q}^n$ is the optimal constant in this inequality 
for all exponents and  all dimensions.

Write $\bp = (p_1,p_2,p_3)$ with $p_j\in[1,\infty]$,
$\bff=(f_1,f_2,f_3)$, and $\bx = (x_1,x_2,x_3)$ where each $x_j\in\reals^m$.
We use the notational convention
\begin{equation} \norm{\bff}_\bp = \prod_{j=1}^3\norm{f_j}_{p_j}. \end{equation}
An ordered triple $\bp=(p_1,p_2,p_3)$ of exponents is said to be admissible
if $p_j\in[1,\infty]$  and $\sum_{j=1}^3 p_j^{-1}=2$.

Rather than work with the bilinear operation $(f,g)\mapsto f*g$, we will work with
the trilinear form
\begin{equation}
\scriptt(\bff) = \scriptt_{\reals^m}(\bff) =  \int_{x_1+x_2+x_3=0} 
\prod_{j=1}^3 f_j(x_j)\,d\lambda_{\reals^m}(\bx)
\end{equation}
where 
$\lambda_{\reals^m}$ is the natural Lebesgue measure on 
\begin{equation} \label{eq:LambdaRm} \Lambda_{\reals^m} = \{\bx\in(\reals^m)^3: x_1+x_2+x_3=0\}.
\end{equation}
That is,
\[\lambda_{\reals^m}(E) = \int_{\reals^{m}\times\reals^{m}} \one_E(x_1,x_2,-x_1-x_2)\,dx_1\,dx_2.\]
The three variables $x_1,x_2,x_3$ may be freely permuted in the discussion
of $\lambda_{\reals^m}$.

For $\bp\in[1,\infty]^3$ 
define the constant \begin{equation} \bestA_\bp = \prod_{j=1}^3 p_j^{1/2p_j}q_j^{-1/2q_j}\end{equation}
where $q_j$ is the exponent conjugate to $p_j$, with $\infty^{\pm 1/\infty}$ interpreted as $1$.
Then $\bestA_\bp$ is strictly less than $1$ whenever $\bp$ is admissible and each $p_j$ belongs
to the open interval $(1,\infty)$.
The inequality of Beckner and Brascamp-Lieb can be restated as
\begin{equation}\label{triyoungreal} \big|\scriptt_{\reals^m}(\bff)\big| \le \bestA_\bp^{m} 
\norm{\bff}_\bp \end{equation}
whenever $\bp$ is admissible.
The factor $\bestA_\bp^m$ is optimal for all exponents.

By a Gaussian function $G$ with domain equal to a Euclidean space $\reals^m$
we mean a function \begin{equation}\label{Gaussianform} G(x) = ce^{-|L(x-a)|^2+i x\cdot b}\end{equation}
where $c\in\complex$, $a\in\reals^m$, $b\in\reals^m$,
and $L:\reals^m\to\reals^m$ is an invertible linear endomorphism.
A linear imaginary term, $ix\cdot b$, is allowed in the exponent,
but the quadratic part of the exponent is real.
In other contexts, the term ``Gaussian'' may refer to functions that are either more, or less, general. 

For the Euclidean group $\reals^m$, 
extremizing triples $\bff$ for Young's convolution inequality exist
for all admissible exponent triples $\bp$ with each $p_j\in(1,\infty)$. 
All such triples were characterized by Brascamp and Lieb \cite{brascamplieb}.
For each admissible 
$\bff\in L^{p_1}\times L^{p_2}\times L^{p_3}$ 
there exists $\gamma(\bp) = (\gamma_1,\gamma_2,\gamma_3)
\in (0,\infty)^3$ with the following property.
Suppose that $\norm{f_j}_{p_j}>0$ for each index $j$.
If $|\scriptt_{\reals^m}(\bff)| = \bestA_\bp^{m}\norm{\bff}_\bp$
then each function $f_j$ is a Gaussian function 
$G_j = c_j e^{-\rho_j |L_j(x-a_j)|^2 + ix\cdot b_j}$. Moreover, 
the ordered triple $(G_1,G_2,G_3)$ is compatible in the sense that
$a_1+a_2+a_3=0$, $b_1=b_2=b_3$, $L_1=L_2=L_3$,
and $\rho_i/\rho_j = \gamma_i/\gamma_j$ for all $i,j\in\{1,2,3\}$.
% and the ratios $\gamma_i/\gamma_j$ are constants that depend only on $\bp$ and the indices $i,j$.
Conversely, if each $f_j$ is Gaussian and if these functions are compatible in the
sense indicated, then
$|\scriptt_{\reals^m}(\bff)| = \bestA_\bp^{m}\norm{\bff}_\bp$.
% ??? fill in details
$\gamma(\bp)$ is uniquely specified by $\bp$ if one requires that $\gamma_1=1$.

A yet sharper formulation of Young's inequality for $\reals^m$ is developed in \cite{christyoungest}.
If $\norm{f_j}_{p_j}=1$ for each index $j$ and if 
$\scriptt(\bff) \ge \bestA^m_{\bp}-\delta$
then $\bff$ lies within distance $\eps(\delta)$ of an extremizing 
triple of Gaussians, in the sense that $\norm{f_j-G_j}_{p_j}\le\eps(\delta)$,
and $\eps(\delta)\to 0$ as $\delta\to 0$.
For a partial range of admissible exponents $\bp$, this is shown
\cite{christHY} to hold with $\eps(\delta) = C(m,\bp)\delta^{1/2}$.

\subsection{Young's inequality for Heisenberg groups}

Let $d\in\naturals$, and
identify $\reals^{2d+1}$ with $\reals^{2d}\times\reals$.
The Heisenberg group $\heis^d$ is $\reals^{2d+1}$ as a set, with the group law
\begin{equation} z\cdot z' = (x,t)\cdot(x',t') = (x+x',\,t+t'+\sigma(x,x')) \end{equation}
where $z=(x,t)$, $z'=(x',t')$, and $\sigma:\reals^{2d}\times\reals^{2d}\to\reals^1$
is the symplectic form 
\begin{equation} \sigma(x,x') = \sum_{j=1}^d \big(x_j x'_{j+d} -x_{j+d}x'_j\big).  \end{equation}
Although we use multiplicative notation for the group law, we denote the
the group identity element by  $0 =(0,0)$.
The Heisenberg multiplicative inverse of $(x,t)$ is $(-x,-t)$.
There are of course many alternative isomorphic formulations of this group law,
some of which are in common use.
By a Gaussian function $G:\heis^d\to\complex$ we mean a Gaussian function $G:\reals^{2d+1}\to\complex$,
with respect to the coordinate system for $\heis^d$ introduced above.

$L^p$ norms on $\heis^d$ are defined with respect to Lebesgue measure
on $\reals^{2d+1}$, and will be denoted by $\norm{\cdot}_{L^p}$
and more succinctly by $\norm{\cdot}_p$.
Throughout this paper, integrals over $\heis^d$ or
subsets of $\heis^d$ measure are understood to be with respect
to Lebesgue measure, unless the contrary is explicitly indicated.
Convolution is defined to be
$f*g(u) = \int_{\heis^d} f(uv^{-1})g(v)\,dv$. 
This bilinear operation is associative, but not commutative, on the Schwartz space.

We phrase Young's inequality for $\heis^d$ in terms of the trilinear form
\begin{equation}
\scriptt_{\heis^d}(\bff) =  \int_{z_1z_2z_3=0} \prod_{j=1}^3 f_j(z_j)\,d\lambda(\bz)
\end{equation}
where $z_1z_2z_3$ is the threefold $\heis^d$ product
and $\lambda=\lambda_{\heis^d}$ is the natural Lebesgue measure on 
\begin{equation} \label{eq:LambdaHd}
\Lambda_{\heis^d} = \{\bz\in(\heis^d)^3: z_1z_2z_3=0\}.
\end{equation}
That is, 
\[\lambda(E) = \int_{\heis^{d}\times\heis^{d}} \one_E(z_1,z_2,z_2^{-1}z_1^{-1})\,dz_1\,dz_2\]
and the roles of the variables $z_1,z_2,z_3$ can be interchanged provided that noncommutativity
of the group law is taken properly into account.
Recall that the group identity element of $\heis^d$ is denoted by $0$.
Just as in the Euclidean case, it is elementary that $|\scriptt_{\heis^d}(\bff)| \le \norm{\bff}_\bp$
whenever $f_j\in L^{p_j}$ for all $j$ and $\bp$ is admissible.
% This is equivalent to $\norm{f*g}_r\le \norm{f}_p\norm{g}_q$ whenever $r^{-1} = p^{-1}+q^{-1}-1$.

Klein and Russo \cite{kleinrusso} and
Beckner \cite{beckner1995} have observed that the sharper inequality
\begin{equation}\label{triyoung} 
\big|\scriptt_{\heis^d}(\bff)\big|\le \bestA_\bp^{2d+1}\norm{\bff}_\bp \end{equation}
holds, with the same constant factor on the right-hand side as for Euclidean space of dimension $2d+1$. 
Moreover, $\bestA_\bp^{2d+1}$ is the optimal constant in this inequality.
Beckner has observed further that there exist no extremizing functions,
that is, $|\scriptt_{\heis^d}(\bff)|$ is strictly less than $\bestA_{\bp}^{2d+1}\norm{\bff}_\bp$
whenever all three functions have positive norms.\footnote{Klein and Russo do not explicitly
discuss existence of extrenizers for Young's inequality, but do prove a closely
related result: There exist no nonzero extremizers for the Heisenberg group
analogue of the $L^p\to L^{p'}$ Hausdorff-Young inequality when the conjugate
exponent $p'$ is an even integer.}

The nonexistence of extremizing functions can be viewed differently.
For each $s\in\reals$, the set $\reals^{2d+1}$ is a group under the operation
$+_s$ defined by
\begin{equation}
(x,t)\,+_s\,(x',t') = (x+x',t+t'+s\sigma(x,x')).
\end{equation}
This group is isomorphic to $\heis^d$ if $s\ne 0$, and to the Euclidean group $\reals^{2d+1}$
for $s=0$.
Haar measure is Lebesgue measure in these coordinates, for all $s$.
The optimal constant in Young's convolution inequality is $\bestA_\bp^{2d+1}$ 
for every $s$.  A datum $(\bff,s)$ realizes this optimal constant if and only if $s=0$
and $\bff$ is a maximizing ordered triple $\bG$ for $\reals^{d+1}$.
Theorem~\ref{thm:main}, below, could be reformulated as an assertion that
$(\bff,s)$ nearly realizes the optimal constant only if $(\bff,s)$
is cloxse to such a datum $(\bG,0)$, in an appropriate sense.

In a series of papers 
\cite{christmarcos},\cite{christyoungest},\cite{christbmtwo},\cite{christbmhigh},\cite{christRS3},\cite{christRShigh},\cite{christHY},\cite{christRSult}
we have studied various sharp inequalities for which extremizing functions (respectively ordered tuples 
of functions or sets) exist and have previously
been characterized. We have shown that functions (respectively ordered tuples of functions or sets)
that nearly extremize the inequalities are nearly equal, in appropriate norms or other measures
of approximation, to extremizing functions (respectively ordered tuples of functions or sets).
The present paper characterizes ordered triples of functions that nearly extremize
Young's inequality for Heisenberg groups --- despite the nonexistence of exact extremizers.

%We have in mind a potential application of the type of result developed here and
%the underlying techniques, but 
%this application would require substantial further development and is left for the future.

\medskip
\noindent{\em Acknowledgements.}\
The author is grateful to Anthony Carbery for pointing out the question addressed here,
to Detlef M\"uller for calling his attention to the reference \cite{kleinrusso},
and to Edward Scerbo for useful comments on the exposition.
He thanks Joe Wolf, as well as Professors Carbery and M\"uller,
for stimulating conversations.

% marker

\section{Definitions and main theorem}
Our main result will state that if $\bff$ nearly extremizes Young's inequality
for $\heis^d$ then there exists an ordered triple $(G_1,G_2,G_3)$ of Gaussians
with certain properties, such that $\norm{f_j-G_j}_{p_j}$ is small for each index $j$.
In order to formulate this result precisely, several definitions are required. 

\subsection{The symplectic group}
Denote by $\sp(2d)$ the symplectic group of all invertible linear mappings
$S:\reals^{2d}\to\reals^{2d}$ satisfying
\begin{equation}
\sigma(Sx,Sx')=\sigma(x,x')\ \text{ for all $x,x'\in\reals^{2d}$}. 
\end{equation}
%Write $\bz = (z_1,z_2,z_3)$ for $z_j\in\heis^d$, and $z_j=(x_j,t_j)\in\reals^{2d}\times\reals^1$.
To $S\in\sp(2d)$ is asociated the group automorphism $(x,t)\mapsto (Sx,t)$ of $\heis^d$.

% Let $M,L$ be $2d\times 2d$ real matrices with $L$ invertible.
Let $J$ denote the $2d\times 2d$ matrix 
\begin{equation} J = \begin{pmatrix} 0 & I \\ -I & 0 \end{pmatrix} \end{equation}
where $I$ is the $d\times d$ identity matrix. 
Since $\sigma(x,y) = \langle x,\,Jy\rangle$ for $x,y\in\reals^{2d}$,
the identity $\sigma(Sx,Sy)\equiv \sigma(x,y)$
that defines $\sp(2d)$ is equivalent to
$\langle Sx,JSy\rangle \equiv \langle x,Jy\rangle$.
Thus $S\in\sp(2d)$ if and only if $S^*JS=J$.

\subsection{Symmetries} \label{section:symmetry}

Let $\Psi=(\psi_1^*,\psi_2^*,\psi_3^*)$ be an ordered $3$--tuple of
invertible linear mappings $\psi_j^*:L^{p_j}(\heis^d)\to L^{p_j}(\heis^d)$.
Consider the functional
\begin{equation}\label{Phiratio} 
\Phi(\bff) = |\scriptt_{\heis^d}(\bff)|\,\norm{\bff}_\bp^{-1},
% \Phi(\bff) = \frac{|\scriptt_{\heis^d}(\bff)|}{\norm{\bff}_\bp}, 
\end{equation} 
defined for all $\bff$ satisfying $\norm{\bff}_\bp\ne 0$.
Given $\bp$, we say that $\Psi$ is a symmetry of the inequality \eqref{triyoung},
or of the functional $\Phi$,
if $\Phi(\Psi\bff)= \Phi(\bff)$ for all $\bff\in L^{p_1}\times L^{p_2}\times L^{p_3}$
with $\norm{\bff}_\bp\ne 0$.
These $3$-tuples form a group under componentwise composition.

%In Definition~\ref{def:symmetrygroup} 
%of \S\ref{section:symmetry}, below, we introduce a particular group $\fg(\heis^d)$ of such
%symmetries. This group is used in the formulation of our main theorem, which we state
%before presenting the definition of $\fg(\heis^d)$.

%In this section we define the symmetry group $\fg(\heis^d)$.

% We first introduce five types of symmetries.

%Write $\bff\circ\Psi = (f_j\circ\psi_j)_{1\le j\le 3}$.
%We say that $\Psi$ is a symmetry of the inequality \eqref{triyoung} if
%\begin{equation} \Phi_\bp(\bff\circ\Psi) = \Phi_\bp(\bff)
%\ \text{ for all $0\ne f_j \in L^{p_j}$.} \end{equation}

Most of the symmetries of $\Phi$ relevant to our considerations
are defined in terms of mappings of the underlying space $\heis^d$.
To any diffeomorphism $\psi$ of $\heis^d$ we associate a linear operator
on functions $f:\heis^d\to\complex$, defined by \[\psi^*(f) = f\circ\psi.\]
We next list four families of ordered triples $(\psi_1,\psi_2,\psi_3)$
of diffeomorphisms of $\heis^d$ 
such that $\Psi=(\psi_1^*,\psi_2^*,\psi_3^*)$ is a symmetry of $\Phi$. 
The first three of these families are:
\begin{equation}
\left\{ \ \ \begin{alignedat}{2}
& (i) &\ \ &
\psi_j(x,t) = 
(rx,r^2 t)
\ \text{ with $r\in\reals^+$} 
\\
& (ii) &&
\psi_j(z) = 
(u_j z w_j)
\ \text{ with $w_1 = u_2^{-1}$, $w_2=u_3^{-1}$, and $w_3=u_1^{-1}$.}
\\ 
& (iii) &&
\psi_j(x,t) = 
(Sx,t)
\text{ with $S\in\sp(2d)$.}
\end{alignedat}\right.
\end{equation}
The fourth family is defined by
\begin{equation} \psi_j(x,t) = (x,t+\varphi_j(x)) \end{equation}
where $(\varphi_1,\varphi_2,\varphi_3)$ is an ordered triple of affine mappings from $\reals^{2d}$
to $\reals^1$ that satisfies 
$\sum_{k=1}^3 \varphi_k(x_k)=0$ whenever $\sum_{k=1}^3 x_k=0$.
In (i), $r$ is independent of $j$; likewise $S$ is independent of $j$ in (iii).
In (ii), $u_jz_jw_j$ is the $\heis^d$ group product of these three elements.

A fifth family of symmetries is defined in terms of modulations of functions,
rather than diffeomorphisms of the underlying space.
For any $u\in \reals^{2d}$ define
$\Psi = (\psi_1,\psi_2,\psi_3)$ by 
\begin{equation}
(\psi_j f)(x,t) = e^{iu\cdot x}f(x,t).
\end{equation}
The exponent $iu\cdot x$ depends only on the coordinate $x$, not on $t$.

Each component of each element of each of these five families is an invertible bounded
linear operator on $L^p(\heis^d)$ for all $p\in[1,\infty]$.
By the composition $\Psi\circ\Psi'$ of two such ordered triples we mean 
the ordered triple $(\psi_1\circ\psi'_1,\psi_2\circ\psi'_2,\psi_3\circ\psi'_3)$
defined by componentwise composition. 

\begin{lemma}
Each of the ordered triples of linear operators $\Psi$ 
listed above is a symmetry of the ratio $\Phi$ for every admissible $\bp$.
\end{lemma}
The straightforward verifications are left to the reader. \qed

\begin{definition} \label{def:symmetrygroup}
$\fg(\heis^d)$ denotes the group of all ordered triples $\Psi$ of diffeomorphisms of $\heis^d$
that can be expressed as compositions of finitely many symmetries of the inequality \eqref{triyoung},
with each factor being one of the five types introduced above.
\end{definition}

\subsection{Special ordered triples of Gaussians on $\heis^d$}

\begin{comment}
\begin{definition}
Let $\bp$ be admissible.
An ordered triple $\bG = (G_1,G_2,G_3)$ of Gaussian functions
$G_j:\heis^d\to\complex$ is $\bp$--compatible if
there exist $\Psi\in\fg(\heis^d)$,
scalars $c_j\in\reals^+$,
and \fbox{?}
such that for each index $j$,
$G_j = \tilde G_j\circ\psi_j$
where
$\tilde G_j(\bx) =  \prod_{k=0}^{2d} e^{-\rho_{j,k} x_k^2}$ 
with $\rho_{i,k}/\rho_{j,k} = \gamma_i/\gamma_j$
for all $i,j\in\{1,2,3\}$ and all $k\in\{0,1,2,\dots,2d\}$.
\end{definition}
\end{comment}

\begin{definition}
Let $d\ge 1$ and $\eps>0$. 
A canonical $\eps$--diffuse Gaussian is a function $G:\heis^d\to\complex$ of the form
\[G(x,t) =  e^{-|Lx|^2} e^{-a t^2} e^{ibt} \]
% $e^{iv\cdot x}$ where $v\in\reals^{2d}$,
where $a>0$, $b\in\reals$, and $L:\reals^{2d}\to\reals^{2d}$ is an invertible linear endomorphism,
which together satisfy
\begin{equation} 
\max(a^{1/2},a,\,|b|) \cdot \norm{L^{-1}}^2 \le\eps.
\end{equation}
\end{definition}

Recall the ordered triple $\gamma(\bp)$ introduced above in the
discussion of maximizers for Young's inequality for $\reals^m$.
\begin{definition}
Let $\bp$ be admissible.
An ordered triple $\bG = (G_1,G_2,G_3)$
of canonical $\eps$--diffuse Gaussians
\[G_j(x,t) =  e^{-|L_j x|^2} e^{-a_j t^2} e^{ib_jt} \]
is said to be $\bp$--compatible if
% $v_1=v_2=v_3\in\reals^{2d}$ and
there exist $L,a,b$ such that
$L_j = \gamma_j^{1/2}L$,
$a_j = \gamma_j a$,
and $b_j=b$
for all $j\in\{1,2,3\}$.
\end{definition}

\begin{definition}
Let $d\ge 1$ and let $\eps>0$ be small.
An ordered triple $\bG=(G_1,G_2,G_3)$ of  Gaussian functions
$G_j:\heis^d\to\complex$ is $\eps$--diffuse and $\bp$--compatible if
there exist $\Psi\in\fg(\heis^d)$, scalars $c_j\in\reals^+$,
and a $\bp$--compatible ordered triple $(\tilde G_1,\tilde G_2,\tilde G_3)$
of canonical $\eps$--diffuse Gaussian functions 
such that \[G_j = c_j \psi_j^*\tilde G_j\ \text{ for each index $j \in\{1,2,3\}$.} \]
\end{definition}

\subsection{Main theorem}
\begin{theorem}\label{thm:main}
For each $d\ge 1$ and each 
admissible ordered triple $\bp$ of exponents there exists a function $\delta\mapsto\eps(\delta)$
satisfying $\lim_{\delta\to 0} \eps(\delta)=0$ with the following property.
Let $\bff\in L^{\bp}(\heis^d)$ and suppose that $\norm{f_j}_{p_j}\ne 0$ for each $j\in\{1,2,3\}$.
Let $\delta\in(0,1)$ and
suppose that $|\scriptt_{\heis^d}(\bff)| \ge (1-\delta)\bestA_\bp^{2d+1}\norm{\bff}_{\bp}$.
Then there exists a  $\bp$--compatible
$\eps(\delta)$--diffuse
ordered triple of Gaussians $\bG = (G_1,G_2,G_3)$
such that \begin{equation} \norm{f_j-G_j}_{p_j} < \eps(\delta)\norm{f_j}_{p_j}\ 
\text{ for $j\in\{1,2,3\}$.} \end{equation}
% Moreover, each Gaussian $G_j$ is $\eps(\delta)$--diffuse.
\end{theorem}

Thus $G_j = c_j\psi_j^*\tilde G_j$
where $c_j\in\complex$, $(\tilde G_1,\tilde G_2,\tilde G_3)$
is a canonically $\eps(\delta)$--diffuse $\bp$--compatible ordered triple
of Gaussians, and $\Psi = (\psi_1,\psi_2,\psi_3)\in\fg(\heis^d)$.
All five types of elements of $\fg(\heis^d)$ are encountered in the analysis.

The technique developed here has been adapted to the $ax+b$ group,
and an analogue of Theorem~\ref{thm:main} for that group has been established,
by E.~Scerbo \cite{scerbo}.

\section{Approximate solutions of functional equations}

A principal ingredient of the analysis is a quantitative expression of 
the unsolvability of a variant of the functional equation
\begin{equation}\label{functlequation} \varphi(x)+\psi(y)+\xi(x+y)=0.\end{equation} 
This variant takes the form
\begin{equation} \label{functlequationvariant}
\varphi(x)+\psi(y)+\xi(x+y) + \sigma(x,y)=0
\end{equation} 
%where $\sigma_L(x,y) = \sigma(L^{-1}x,L^{-1}y)$,
%$L:\reals^{2m}\to\reals^{2m}$ is an invertible linear transformation,
%and 
where the functions $\varphi,\psi,\xi$ have domains equal to $\reals^{2m}$. 
Its unsolvability is formulated below, in quantitative terms, as Proposition~\ref{prop:applytoaj}.

An {\it ad hoc} argument that relies on the antisymmetry of
$\sigma(x,y)$ will enable us to deduce the information needed
concerning \eqref{functlequationvariant} from what is already known
about approximate solutions of \eqref{functlequation}.
This leads naturally to analogous questions about
more general functional equations, for which this {\it ad hoc} argument may not apply.
We therefore digress to present the following general result, 
which is suggested and motivated by considerations
in this paper, but is not actually used in the proofs of the main theorems.

Consider the difference operators
\begin{equation} \Delta_h f(x) = f(x+h)-f(x),\end{equation}
where $x\in\reals^d$ and $+$ denotes the Euclidean group operation.
Let $\ball$ be an arbitrary ball of positive, finite radius
in $\reals^d$ and let $\tilde\ball$ be a ball of positive,
finite radius in $\reals^d$ centered at the origin.

\begin{theorem} \label{thm:polyfunctleqn}
For each dimension $d\ge 1$, each nonnegative integer $D$,
and each $\eta>0$ 
there exists a function $\delta\mapsto\eps(\delta)$
satisfying $\lim_{\delta\to 0} \eps(\delta)=0$ with the following property.
Suppose that $|\tilde\ball|\ge\eta|\ball|$, $0<\delta\le1$,
and $A\in[0,\infty)$.
Let $\varphi:\ball+\tilde\ball\to\complex$ be Lebesgue measurable.
Suppose that there exists a function
$\ball\times \tilde\ball \owns (x,h) \mapsto P_h(x)\in\complex$
such that
\begin{equation}\label{eq:polyfun1} \big|\Delta_h \varphi(x) - P_h(x)\big| \le A \end{equation}
for all $(x,h)\in\ball\times \tilde\ball$
with the exception of a set of measure $\le \delta|\ball|\cdot |\tilde\ball|$.
Suppose that
%Suppose that for all $(x,h)\in\ball\times \tilde\ball$
%with the exception of a set of measure $\le \delta|\ball|\cdot |\tilde\ball|$,
%\begin{equation}\label{eq:polyfun1} \big|\Delta_h \varphi(x) - P_h(x)\big| \le A \end{equation} where
\begin{equation}\label{eq:polyfun2} P_h(x) = \sum_{|\alpha|\le D} a_\alpha (h)x^\alpha \end{equation}
is a polynomial function of $x$ of degree $\le D$
whose coefficients $a_\alpha$ are Lebesgue measurable functions of $h$.
Then there exists a polynomial $Q$ of degree at most $D+1$ such that
\begin{equation}\label{eq:polyfunconclusion} \big|\varphi(x)-Q(x)\big| \le CA \end{equation} 
for all $x\in\ball$ outside a set of measure $\le\eps(\delta)|\ball|$.
The constant $C$ and function $\eps$ depend only on $d,D,\eta$. 
\end{theorem}

This is proved in \S\ref{section:differencerelations}.
In the simplest case $D=0$, the assumption is that $|\varphi(x+h)-\varphi(x)- a(h)|\le A$
for nearly all points of $\ball\times\tilde\ball$; one has an approximate version
of the fundamental functional equation \eqref{functlequation}. In that special case,
Theorem~\ref{thm:polyfunctleqn} is proved in \cite{christyoungest}.

It is natural to also record a multiplicative analogue the preceding theorem.
\begin{theorem} \label{thm:polyfunctleqnmult}
For each dimension $d\ge 1$, each nonnegative integer $D$,
and each $\eta>0$ 
there exists a function $\delta\mapsto\eps(\delta)$
satisfying $\lim_{\delta\to 0} \eps(\delta)=0$ with the following property.
Suppose that $|\tilde\ball|\ge\eta|\ball|$, $0<\delta\le1$,
and $A\in[0,2]$.
Let $\varphi:\ball+\tilde\ball\to\reals$ be Lebesgue measurable.
Suppose that there exists a function
$\ball\times \tilde\ball \owns (x,h) \mapsto P_h(x)\in\reals$
such that
\begin{equation}\label{eq:polyfun1mult} 
|e^{i(\varphi(x+h)-\varphi(x))}e^{-iP_h(x)}-1|
%\big|\Delta_h \varphi(x) - P_h(x)\big| 
\le A 
\end{equation}
for all $(x,h)\in\ball\times \tilde\ball$
with the exception of a set of measure $\le \delta|\ball|\cdot |\tilde\ball|$.
Suppose that
%Suppose that for all $(x,h)\in\ball\times \tilde\ball$
%with the exception of a set of measure $\le \delta|\ball|\cdot |\tilde\ball|$,
%\begin{equation}\label{eq:polyfun1} \big|\Delta_h \varphi(x) - P_h(x)\big| \le A \end{equation} where
\begin{equation}\label{eq:polyfun2} P_h(x) = \sum_{|\alpha|\le D} a_\alpha (h)x^\alpha \end{equation}
is a polynomial function of $x$ of degree $\le D$
whose coefficients $a_\alpha$ are Lebesgue measurable real-valued functions of $h$.
Then there exists a polynomial $Q$ of degree at most $D+1$ such that
\begin{equation}\label{eq:polyfunconclusion} 
|e^{i\varphi_x)}e^{-iQ(x)}-1|\le CA
% \big|\varphi(x)-Q(x)\big| \le CA 
\end{equation} 
for all $x\in\ball$ outside a set of measure $\le\eps(\delta)|\ball|$.
The constant $C$ and function $\eps$ depend only on $d,D,\eta$. 
\end{theorem}

% We require the following generalization of Lemma~\ref{lemma:simplefunctleqn}.

\section{Analogue for twisted convolution}

Consider twisted convolution of functions with domains $\reals^{2d}$.
The associated trilinear forms are
\begin{equation}
\scriptt_{\reals^{2d},\lambda}(\bff)
= \int_{(\reals^{2d})^3} 
e^{i\lambda \sigma(x_1,x_2)}
\prod_{j=1}^3 f_j(x_j) \,d\lambda_{\reals^{2d}}(\bx)
\end{equation}
where $0\ne\lambda\in\reals$ is a parameter and $\bx=(x_1,x_2,x_3)$.
Since $|\scriptt_{\reals^{2d},\lambda}(\bff)|\le \scriptt_{\reals^{2d}}(|f_1|,|f_2|,|f_3|)$,
one has
\begin{equation}
| \scriptt_{\reals^{2d},\lambda}(\bff)|\le \bestA_{\bp}^{2d} \prod_{j=1}^3\norm{f_j}_{p_j}
\end{equation}
for admissible $\bp$.
The constant $\bestA_{\bp}^{2d}$ is optimal \cite{kleinrusso}, 
as one sees by considering ordered triples of
Gaussians that extremize Young's inequality for $\reals^{2d}$
and are concentrated near $0$.
Again, there exist no extremizing triples \cite{kleinrusso}.

%\begin{definition}
%\end{definition}

\begin{theorem} \label{thm:alt}
For each $d\ge 1$ and each 
admissible ordered triple $\bp$ of exponents there exists a function $\delta\mapsto\eps(\delta)$
satisfying $\lim_{\delta\to 0} \eps(\delta)=0$ with the following property.
Let $\bff\in L^{\bp}(\reals^{2d})$ and suppose that $\norm{f_j}_{p_j}\ne 0$ for each $j\in\{1,2,3\}$.
Let $\delta\in(0,1)$ and
suppose that $|\scriptt_{\reals^{2d},\lambda}(\bff)| \ge (1-\delta)\bestA_\bp^{2d}\norm{\bff}_{\bp}$.
Then there exist $S\in\sp(2d)$ and a  $\bp$--compatible
ordered triple of Gaussians $\bG = (G_1,G_2,G_3)$
such that $G_j^\natural=G_j\circ S$ satisfy
\begin{equation} \norm{f_j-G_j^\natural}_{p_j} < \eps(\delta)\norm{f_j}_{p_j}\ 
\text{ for $j\in\{1,2,3\}$} \end{equation}
and $G_j$ take the form 
\begin{equation} G_j(x) = c_j e^{-\gamma_j(\bp)|L(x-a_j)|^2}e^{ix\cdot v} \end{equation}
where $v\in\reals^{2d}$, $0\ne c_j\in\complex$,
$a_1+a_2+a_3=0$,
and 
\begin{equation} |\lambda|\cdot \norm{L^{-1}}^2\le\eps(\delta).  \end{equation}
\end{theorem}

The proof of this theorem follows that of Theorem~\ref{thm:main}, with some simplifications.
Details are left to the reader.

\section{Nonexistence of extremizers and value of the optimal constant} \label{section:recap}

We begin by reviewing proofs that the optimal constant
in Young's inequality for $\heis^d$ equals the optimal constant for
Euclidean space of dimension $2d+1$,
and that extremizing triples do not exist.
To show that the constant for $\heis^d$ is at least as large as for $\reals^{2d+1}$,
let $\eps>0$ be small, and consider the ordered triple 
of functions
$\bff_\eps = (f_{j,\eps}: 1\le j\le 3)$ with $f_{j,\eps}(x,t) = e^{-\gamma_j |x|^2} e^{-\eps \gamma_j t^2}$
and $\gamma(\bp) = (\gamma_1,\gamma_2,\gamma_3)$.
For each $\eps>0$, $\bff_\eps$ extremizes Young's inequality for $\reals^{2d+1}$.
One finds by a simple change of variables $t = \eps^{-1/2}s$ that
\begin{equation} \frac{\scriptt_{\heis^{d}}(\bff_\eps)}{\scriptt_{\reals^{2d+1}}(\bff_\eps)}
\to 1\ \text{ as $\eps\to 0$.} \end{equation}

To prove the reverse implication, let $f_j\in L^{p_j}(\heis^d)$ be nonzero nonnegative functions
which are otherwise arbitrary.
Define
\begin{equation} \label{eq:Fjdefn} \left\{ \begin{aligned}
& F_j(x) = \norm{f_j(x,\cdot)}_{L^{p_j}(\reals)}
\\&
f_{j,x}(t) = f_j(x,t)/F_j(x) \ \text{ if $F_j(x)\ne 0$},
\end{aligned} \right. \end{equation}
with instead  $f_{j,x}(t) \equiv 0$ if $F_j(x)=0$.
Write $\bx = (x_1,x_2,x_3)$. 
Then 
\begin{equation}\label{eq:Eucreduction}
\scriptt_{\heis^d}(\bff) = \int_{\Lambda_{\reals^{2d}}} \prod_{j=1}^3 F_j(x_j)
\scriptt_{\reals^1}(f_{1,x_1},f_{2,x_2},f^\dagger_{3,\bx})\,d\lambda_{\reals^{2d}}(\bx)
\end{equation}
where
\begin{equation} f^\dagger_{3,\bx}(s) = f_{3,x_3}(s + \sigma(x_1,x_2)). \end{equation}
Straightforward calculation gives
$f_{3,x_3}(s + \sigma(x_1,x_2)+\sigma(x_1+x_2,x_3))$ 
as the natural definition of $f^\dagger_{3,\bx}(s)$,
but outside of a $\lambda_{\reals^{2d}}$--null set 
this simplifies to $f_{3,x_3}(s+\sigma(x_1,x_2))$ since 
\[ x_1+x_2+x_3=0 \Longrightarrow \sigma(x_1+x_2,x_3)=\sigma(x_1+x_2,-x_1-x_2)=0.\]

Therefore
\[ |\scriptt_{\reals^1}(f_{1,x_1},f_{2,x_2},f^\dagger_{3,\bx})| 
\le \bestA_{\bp}
\prod_{j=1}^3 \norm{f_{j,x_j}}_{p_j}
\le \bestA_{\bp}
\]
with equality only if $\prod_{j=1}^3 F_j(x_j)\ne 0$
and $(f_{1,x_1},f_{2,x_2},f^\dagger_{3,\bx})$ is an extremizing
triple for Young's inequality for $\reals^1$.
Inserting this into \eqref{eq:Eucreduction} gives
\begin{multline*}
|\scriptt_{\heis^d}(\bff)|  \le \bestA_{\bp} \int_{x_1+x_2+x_3=0}
\prod_{j=1}^3 F_j(x_j)\,d\lambda_{\reals^{2d}}(\bx)
\\ = \bestA_{\bp} \scriptt_{\reals^{2d}}(F_1,F_2,F_3)
\le  \bestA_{\bp} \bestA_{\bp}^{2d} \prod_{j=1}^3 \norm{F_j}_{L^{p_j}(\reals^{2d})}
=   \bestA_{\bp}^{2d+1} \norm{\bff}_\bp.
\end{multline*}
This proves that the optimal constant for $\heis^d$ cannot exceed the optimal constant
for $\reals^{2d+1}$. 

This analysis implicitly proves that extremizers do not exist for $\heis^d$.
For arbitrary nonnegative $f_j\in L^{p_j}(\heis^d)$ with positive norms,
we have shown that
%\[ |\scriptt_{\heis^d}(\bff)| \le \bestA_{\bp} \scriptt_{\reals^{2d}}(F_1,F_2,F_3)
%\le \bestA_{\bp}^{2d+1} \prod_{j=1}^3 \norm{F_j}_{p_j}
%= \bestA_{\bp}^{2d+1} \prod_{j=1}^3 \norm{f_j}_{p_j},\]
equality holds only if both
(i) for $\lambda$--almost every $\bx\in (\reals^{2d})^3$,
$(f_{1,x_1},f_{2,x_2},f^\dagger_{3,x_3})$
is an extremizing triple for Young's inequality for $\reals^1$
and (ii) $(F_1,F_2,F_3)$ is an extremizing triple for Young's inequality for $\reals^{2d}$.

By the characterization of equality in Young's inequality for $\reals^{2d}$,
each $F_j$ must be a Gaussian; in particular, $F_j$ is nonzero almost everywhere.
Likewise,
$f_{j,y}$ must be a Gaussian for almost every $y\in\reals^{2d}$ for each index $j\in\{1,2,3\}$.
Moreover, $(f_{1,x_1},f_{2,x_2},f^\dagger_{3,x_3})$ must be $\bp$--compatible.
Expressing 
\[ f_{j,y}(s)  = c_{j}(y)e^{-\gamma_{j}(y)(s-a_{j}(y) )^2 + ib_{j}(y)s},\]
compatibility forces the functional equation
\begin{equation} \label{twisted} a_1(y_1)+a_2(y_2)+a_3(-y_1-y_2) + \sigma(y_1,y_2)=0 \end{equation}
for almost every $(y_1,y_2)\in\reals^{2d}\times\reals^{2d}$.

\begin{lemma} \label{lemma:nosoln}
There exists no ordered triple of measurable functions
$a_j:\reals^{2d}\to\complex$ that satisfies the functional equation
\eqref{twisted} for almost every $(y_1,y_2)\in(\reals^{2d})^2$.
\end{lemma}

\begin{proof}[Proof of Lemma~\ref{lemma:nosoln}]
Write \eqref{twisted} with the roles of $y_1,y_2$ interchanged, and add the result to \eqref{twisted}.
Since $\sigma$ is antisymmetric, its contributions cancel, leaving
\[ a(x_1)+a(x_2)+a_3(-x_1-x_2)=0\] 
for almost every $(x_1,x_2)\in(\reals^{2d})^3$, where $a = \tfrac12 a_1+ \tfrac12 a_2$.
As is well known, any measurable solutions of this functional
equation must agree almost everywhere with affine functions.
Thus $a_3$ is affine. 

Inserting this conclusion into \eqref{twisted}, we conclude that
there exist functions $\tilde a_j$, which differ from $a_j$ by affine functions,
such that $\tilde a_1(x_1)+\tilde a_2(x_2) + \sigma(x_1,x_2)=0$ almost everywhere.
By freezing almost any value of $x_2$ one finds that $\tilde a_1$
agrees almost everywhere with an affine function. 
The same reasoning applies to $\tilde a_2$.
But the original equation \eqref{twisted} cannot hold with all three functions $a_j$ affine,
since $\sigma$ is not affine.
\end{proof}

This paper establishes a more quantitative form of Lemma~\ref{lemma:nosoln},
and reduces Theorem~\ref{thm:main} to this result by elaborating on the reasoning shown above.
Klein and Russo \cite{kleinrusso} have shown how the same type of reasoning as that shown above
can be applied to certain semidirect product Lie groups. 
%proving that the optimal constant in
%Young's inequality for those groups agrees with the optimal constant for the Euclidean
%group of the same topological dimension, but that no extremizers exist.
Much of the quantitative analysis below extends straightforwardly to more general semidirect products. 
However, each semidirect product leads to its own analogue of the variant \eqref{twisted} 
of the classical functional equation \eqref{functlequation}. In this paper we analyze
only one such variant, leaving a general investigation for future work.
Forthcoming work of E.~Scerbo \cite{scerbo} will adapt this analysis to the $ax+b$ group.
% \eqref{nearlycharacter}.

\begin{remark}
There is no solution $(a_1,a_2,a_3)$ of \eqref{twisted} in the sense of distributions.
This remark does not subsume Lemma~\ref{lemma:nosoln},
since the lack of any assumption in that lemma that the functions $a_j$ are locally integrable
prevents their being interpreted as distributions.

To show this, write $y_j = (y_{j,k})_{1\le k\le 2d}$.
Applying $\frac{\partial^2}{\partial y_{1,m}\partial y_{1,n}}$ gives
\[ \frac{\partial^2 a_1}{\partial y_{1,m}\partial y_{1,n}}(y_1)+
\frac{\partial^2 a_3}{\partial y_{1,m}\partial y_{1,n}}(y_1+y_2)
\equiv 0,\]
whence 
$\frac{\partial^2 a_3}{\partial y_{1,m}\partial y_{1,n}}(y_1+y_2)$
is independent of $y_2$ as a distribution. Therefore $a_3$, and hence $a_1$, 
are quadratic polynomials.  The same applies to $a_2$.

Now consider any $k\in\{1,2,\dots,d\}$ and apply
$\frac{\partial^2}{\partial y_{1,k}\partial y_{2,k+d} }
+\frac{\partial^2}{\partial y_{2,k}\partial y_{1,k+d}}$ 
to both sides of \eqref{twisted}. This differential monomial annihilates
$\sigma(y_1,y_2)$. It results that 
$\frac{\partial^2}{\partial y_{k}\partial y_{k+d}}a_3\equiv 0$.
By applying
$\frac{\partial^2}{\partial y_{1,m}\partial y_{2,n}}$
for other pairs $m,n$ one obtains 
$\frac{\partial^2}{\partial y_{m}\partial y_{n}}a_3\equiv 0$
for all $m,n$.
Thus $a_3$ is an affine function.

Once this is known,
apply to 
$\frac{\partial^2}{\partial y_{1,m}\partial y_{1,n}}$
to conclude that $a_1$ is affine. In the same way, $a_2$ is affine.
\eqref{twisted} now expresses
$\sigma(y_1,y_2)$ as a sum of three affine functions, contradicting
the definition of $\sigma$.
\end{remark}

\begin{comment}
The final term $\sigma(h,x_1+x_2)$
is a linear function of $x_1+x_2$.
By taking differences with respect to $x_1$ we may eliminate the term $-\Delta_ha_2(x_2)$
to conclude that 
$\Delta_h\Delta_k a_1(x_1)$ is independent of $x_1$ for almost every $(h,k)$.
After possibly translating $a_1$ and subtracting a constant we conclude that
\[ a_1(h+k) = a_1(h)+a_1(k)\ \text{ for almost every $(h,k)$.}\]
This implies that $a_1$ agrees almost everywhere with an affine function.
The same reasoning applies to $a_2,a_3$.
The conclusion is that there exists an ordered triple of affine functions $a_j$
that satisfies the functional equation \eqref{twisted} almost everywhere.
But this is impossible; the first three terms on the left-hand side
of \eqref{twisted} are affine functions of $(x_1,x_2)$, while the term $\sigma(x_1,x_2)$
is not such a function. The sum cannot vanish identically.
\end{proof}
\end{comment}

\section{Sufficiency}

\begin{proposition}\label{prop:sufficiency}
Let $d\ge 1$, and let $\bp$ be admissible. For each $\eps>0$
there exists $\eta(\eps)>0$ satisfying $\lim_{\eps\to 0} \eta(\eps)=0$
with the following property.
For any $\bp$--compatible $\eps$--diffuse 
ordered triple $\bG=(G_1,G_2,G_3)$ of Gaussian functions, 
\[\scriptt_{\heis^d}(\bG) \ge (1-\eta(\eps)) \bestA_\bp^{2d+1}\prod_{j=1}^3\norm{G_j}_{p_j}.\]
\end{proposition}

More generally, it follows immediately from the triangle inequality that if $\bG$ is $\bp$--compatible
and $\eps$--diffuse, and if $\norm{f_j-G_j}_{p_j}<\eps\norm{f_j}_{p_j}$
for all $j\in\{1,2,3\}$
then 
\[|\scriptt_{\heis^d}(\bff)| \ge (1-\eta(\eps)) \bestA_\bp^{2d+1}\prod_{j=1}^3\norm{f_j}_{p_j}\]
where the function $\eta$ is modified but is still $o_\eps(1)$.

The following notation will be used throughout the analysis, here and below.
\begin{definition} 
For any invertible linear endomorphism $L$ of $\reals^{2d}$, 
\begin{equation} \label{eqdef:sigmaL}
\sigma_L(x,y) = \sigma(L^{-1}x,L^{-1}y)  \end{equation}
for $x,y\in\reals^{2d}$.
\end{definition}

\begin{proof}[Proof of Proposition~\ref{prop:sufficiency}]
Since the action of $\fg(\heis^d)$ preserves the ratio
$|\scriptt_{\heis^d}(\bff)|/\prod_{j=1}^3\norm{f_j}_{p_j}$,
it suffices to prove this for $\bp$--compatible
ordered triples of canonical $\eps$--diffuse Gaussians.
Thus we may assume that 
\[G_j(x,t) = e^{-\gamma_j |Lx|^2} e^{-\gamma_j a t^2}e^{ib t} \]
where $L$ is an invertible linear endomorphism of $\reals^{2d}$, $a>0$, $b\in\reals$, and
$\max(a^{1/2},|b|)\norm{L^{-1}}^2\le\eps$.  In this situation,
\begin{multline*}
\scriptt_{\heis^d}(\bG)
= \int_{\reals^{2d}\times\reals^{2d}}
e^{-\gamma_1|Lx_1|^2-\gamma_2|Lx_2|^2-\gamma_3|L(x_1+x_2)|^2}
\\
\cdot \int_{\reals\times\reals}
e^{-\gamma_1 at_1^2-\gamma_2 at_2^2 - \gamma_3 a (t_1+t_2 + \sigma(x_1,x_2))^2}
e^{i[bt_1+bt_2-b(t_1+t_2+ \sigma(x_1,x_2))]}
\,dt_1\,dt_2
\,dx_1\,dx_2.
\end{multline*}
Cancelling where possible and substituting $Lx_j=y_j$ gives
$|\det(L)|^{-2}\cdot I$ 
where
\begin{multline*}
I = \int_{\reals^{4d}}
e^{-\gamma_1|y_1|^2-\gamma_2|y_2|^2-\gamma_3|y_1+y_2|^2}
e^{-ib\sigma_L(y_1,y_2)}
\\ \cdot
\int_{\reals^2}
e^{-\gamma_1 at_1^2-\gamma_2 at_2^2 - \gamma_3 a (t_1+t_2+\sigma_L(y_1,y_2))^2}
\,dt_1\,dt_2
\,dy_1\,dy_2.
\end{multline*}

Define
\begin{equation*}
J = \int_{\reals^{4d}}
e^{-\gamma_1|y_1|^2-\gamma_2|y_2|^2-\gamma_3|y_1+y_2|^2}
\int_{\reals^2}
e^{-\gamma_1 at_1^2-\gamma_2 at_2^2 - \gamma_3 a (t_1+t_2)^2}
\,dt_1\,dt_2
\,dy_1\,dy_2.
\end{equation*}
$\bG$ is an extremizing ordered triple for Young's inequality with exponents $\bp$
for $\reals^{2d+1}$, with the same coordinates $(x,t)$. Thus
$J =  |\det(L)|^{2} \bestA_\bp^{2d+1} \prod_{j=1}^3\norm{G_j}_{p_j}$.
Thus it suffices to prove that
\[ |I| \ge (1-o_\eps(1))J.\]

An application of Young's inequality for $\reals^1$ to the inner integral, 
followed by an application Young's inequality for $\reals^{2d}$ to the remaining outer integral,
also reveals that 
$|I| \le |\det(L)|^2 \bestA_\bp^{2d+1}\prod_{j=1}^3 \norm{G_j}_{p_j}$.

Let $\eps\mapsto\rho(\eps)$
be a function that tends to $\infty$ slowly as $\eps\to 0$.
The same reasoning shows that if the integrand in the integral defining $I$
is replaced by its absolute value, then the contribution of the region
$\scriptr = \{(y_1,y_2)\in\reals^{4d}: |(y_1,y_2)|>\rho(\eps)\}$ to the integral is $o_\eps(1)$.
Since $|b|\norm{L^{-1}}^2\le\eps$ by hypothesis,
\[|b\sigma_L(y_1,y_2)| \le |b| \norm{L^{-1}}^2\rho(\eps)^2
\le \eps^{1/2} \ \text{ uniformly for all } (y_1,y_2)\in\reals^{4d}\setminus\scriptr\]
provided that $\rho(\eps)$ is chosen to satisfy $\rho(\eps)\le \eps^{-1/4}$.
Therefore
$|e^{-ib\sigma_L(y_1,y_2)}-1| = O(\eps^{1/2})$
uniformly for all $y\in\reals^{4d}\setminus \scriptr$.
Therefore
\[ I = 
\int_{\reals^{4d}}
e^{-\gamma_1|y_1|^2-\gamma_2|y_2|^2-\gamma_3|y_1+y_2|^2}
\int_{\reals^2}
e^{-\gamma_1 at_1^2-\gamma_2 at_2^2 - \gamma_3 a (t_1+t_2+\sigma_L(y_1,y_2))^2}
\,dt_1\,dt_2
\,dy_1\,dy_2\]
plus $o_\eps(1)$.

Define $\scriptr' = \{(t_1,t_2)\in\reals^{2}: |(t_1,t_2)|>\rho(\eps)\}$. 
By the same reasoning, to complete the proof it suffices to have
\[ e^{-\gamma_3a2(t_1+t_2)\sigma_L(y_1,y_2)}e^{-\gamma_3 a \sigma_L(y_1,y_2)^2}
= 1+o_\eps(1)\]
uniformly for all $(y_1,t_1,y_2,t_2)$ such that $(t_1,t_2)\in\reals^2\setminus\scriptr'$ and 
$(y_1,y_2)\in\reals^{4d}\setminus\scriptr$. This holds because
\begin{align*}
|a(t_1+t_2)\sigma_L(y_1,y_2)|
&\le a\rho(\eps)\norm{L^{-1}}^2 \rho(\eps)^2
\\|a\sigma_L(y_1,y_2)^2| & \le a \norm{L^{-1}}^4 \rho(\eps)^4,
\end{align*}
while it is given that $(a^{1/2}+a) \norm{L^{-1}}^2\le\eps$.
\end{proof}

\section{Two ingredients}

In order to prove Theorem~\ref{thm:main}, we will make the 
steps of the reasoning in \S\ref{section:recap} quantitative.
The following result from \cite{christyoungest}, 
the analogue for $\reals^m$ of our main result for $\heis^d$,
will be the first of two main ingredients in the analysis.

\begin{theorem} \label{thm:youngest}
For each admissible $\bp\in(1,\infty)^3$ and each $m\in\naturals$
there exist
$\gamma(\bp) = (\gamma_1,\gamma_2,\gamma_3)\in\reals^+$ 
and a function $\delta\mapsto\eps(\delta)$ satisfying
$\lim_{\delta\to 0^+}\eps(\delta)=0$
with the following property.
If $0\ne f_j\in L^{p_j}(\reals^m)$
and if $\bff=(f_j)_{1\le j\le 3}$ satisfies
$|\scriptt_{\reals^m}(\bff)| \ge (1-\delta)\bestA_{\bp}^m\norm{\bff}_\bp$
then there exists an ordered triple of Gaussian functions of the form
\begin{equation} G_j(x) = c_j e^{-\gamma_j|L(x)-a_j|^2+ix\cdot b}\end{equation}
where 
$0\ne c_j\in\complex$, 
$ a_j,b\in\reals^m$, 
$ \sum_{j=1}^3 a_j=0$, 
and $L:\reals^m\to\reals^m$ is a linear automorphism,
such that
\begin{equation} \norm{f_j-G_j}_{p_j} \le \eps(\delta)\norm{f_j}_{p_j}  \end{equation}
for each $j\in\{1,2,3\}$.
\end{theorem}

The ordered triple $\gamma(\bp)$ is independent of $m$ but is not uniquely determined in this statement; 
$(t\gamma_1,t\gamma_2,t\gamma_3)$
works equally well for any $t\in\reals^+$ since a common factor can be absorbed into $L,a_j$.
But $\gamma(\bp)$ is uniquely determined with the normalization $\gamma_1(\bp)\equiv 1$,
which we enforce henceforth.

The second ingredient is a quantitative expression of the unsolvability of
a functional equation.
%a quantitative version of Lemma~\ref{lemma:nosoln} concerning the unsolvability of the
%functional equation \eqref{twisted}.
In the discussion that follows, $\ball$ always denotes a ball of finite,
positive radius centered at the origin in $\reals^d$. $\ball^*$ denotes the ball
centered at $0$ whose radius is twice that of $\ball$. Sets of Lebesgue measure
zero are negligible for all considerations that follow, so we do not distinguish
between open and closed balls. 
The Cartesian product $\ball\times\ball$ is denoted by $\ball^2$.
The following two lemmas are established in \cite{christyoungest}.

\begin{lemma}\cite{christyoungest}
\label{lemma:simplefunctleqn}
For each $d\in\naturals$ 
there exist $\delta_0>0$ and a function $t\mapsto\eps(t)$
satisfying $\lim_{t\to 0^+} \eps(t)=0$ such that the following conclusion holds.
Let $A\in[0,\infty)$ and $\delta\in(0,\delta_0]$.
Let $\varphi,\psi:\ball\to\complex$ and $\xi:\ball^*\to\complex$ be Lebesgue measurable.
Suppose that
\[ |\varphi(x) + \psi(y) + \xi(x+y)|\le A\]
for all $(x,y)\in\ball^2$
outside a set of measure $\le \delta|\ball|^2$.
Then there exists an affine function $h$ such that
\begin{equation} \big| \varphi(x)-h(x)\big| \le CA \end{equation}
for all $x\in\ball$ outside a set of measure $\eps(\delta)|\ball|$.
The constant $C$ and function $\eps$ depend only on $d$.
\end{lemma}

In particular, the constants in the conclusions do not depend on $\ball$.
The following multiplicative variant of Lemma~\ref{lemma:simplefunctleqn} 
is also proved in \cite{christyoungest}.
\begin{lemma}\label{lemma:nearlycharacter}
For each dimension $d\ge 1$ there exists a constant $K<\infty$ 
with the following property.
Let $B\subset\reals^d$ be a ball with positive radius, and let $\eta\in(0,\tfrac12]$.
For $j\in\{1,2,3\}$
let $f_j:2B\to\complex$ be Lebesgue measurable functions that vanish only on sets of Lebesgue measure zero. 
Suppose that 
\begin{equation} \big|\set{(x,y)\in B^2: |f_1(x)f_2(y)f_3(x+y)^{-1}-1|>\eta}\big|<\delta |B|^2.  \end{equation}
Then for each index $j$
there exists a real--linear function $L_j:\reals^d\to\complex$ such that 
\begin{equation} \big|\set{x\in B: |f_j(x)e^{-L_j(x)}-1|>K\eta^{1/K} }\big| \le K\delta |B|.  
\end{equation}
\end{lemma}

The next result is concerned with a Heisenberg variant of Lemma~\ref{lemma:simplefunctleqn}.

\begin{proposition}\label{prop:applytoaj}
For each $d\in\naturals$ there exists $C<\infty$ with
the following property. Let $\ball$ be any ball of finite, positive
radius centered at the origin in $\reals^{2d}$.
Let $A<\infty$ and $\eta>0$.
Let $a_j: \ball^*\to\reals$ be Lebesgue measurable. 
Let $L:\reals^{2d}\to\reals^{2d}$ be an invertible linear transformation.
%Let $q:\reals^m\times\reals^m\to\reals$ be an antisymmetric quadratic polynomial.
Suppose that
\begin{equation}\label{approxFE} |a_1(x)+a_2(y)+a_3(x+y) + \sigma_L(x,y)|\le A\end{equation}
for all $(x,y)\in\ball^2$ outside a Lebesgue measurable set of Lebesgue measure $\le \eta|\ball|^2$.
Then there exists $S\in \sp(2d)$ such that
\begin{equation}\label{eq:qsmall}
\norm{SL^{-1}}\le CA^{1/2}|\ball|^{-1/2d}. \end{equation}
%\sup_{x\in\ball} |SL^{-1}(x)| \le CA^{1/2}.
Moreover,
there exist affine functions $\psi_j$ for $j\in\{1,2,3\}$ satisfying
\[\psi_1(x_1)+\psi_2(x_2)+\psi_3(-x_1-x_2)=0\ \text{ for all $(x_1,x_2)\in\reals^{2d}\times\reals^{2d}$} \]
such that
\begin{equation} \label{eq:approxFElast} |a_j(x)-\psi_j(x)|\le CA
\ \text{ for all $x\in\ball$ outside a set of measure $o_\eta(1)|\ball|$.} \end{equation}
\end{proposition}

Recall that $\sigma_L(x,y) = \sigma(L^{-1}x,L^{-1}y)$.
By $\norm{T}$ we mean in \eqref{eq:qsmall} 
the usual norm $\sup_{0\ne x\in\reals^{2d}} |T(x)|/|x|$.
The main conclusion is that \eqref{approxFE} cannot hold, unless $L$ satisfies
$\inf_{S\in\sp(2d)}\norm{SL^{-1}}=O(|\ball|^{-1/2d}A^{1/2})$.
Moreover, if \eqref{approxFE} does hold, then
$|\sigma_L(x,y)| \le CA$ for all $(x,y)\in\ball^2$;
consequently this term can be dropped from \eqref{approxFE} to yield
$|a_1(x)+a_2(y)+a_3(x+y)|\le CA$.
The conclusion \eqref{eq:approxFElast} follows from this by Lemma~\ref{lemma:simplefunctleqn}.

\begin{proof}[Proof of Proposition~\ref{prop:applytoaj}]
It is given that
\[|a_1(x)+a_2(y) + a_3(x+y) + \sigma(Lx,Ly)|\le A\]
for all $(x,y)\in\ball^2$ outside a set of measure $\le\eta|\ball|^2$.
By interchanging the roles of $x,y$, adding the resulting inequality
to this one, 
%\[|a_1(y)+a_2(x)+a_3(x+y) + \sigma(Ly,Lx)|\le A\]
%for $(x,y)$ satisfying the same restrictions. Adding this to the
%given inequality 
and invoking the antisymmetry of $\sigma$, we conclude that 
\[ \big| \tilde a(x)+\tilde a(y)+a_3(x+y)\big|\le A\]
for all $(x,y)\in\ball$ outside a set of measure $\le C\eta|\ball|^2$,
where $2\tilde a = a_1 + a_2$.  By 
Lemma~\ref{lemma:simplefunctleqn}
this implies that there exists an affine function $\psi_3$
such that $|a_3(x)-\psi_3(x)|$ for all $x\in \ball$ outside a set of measure $\le C\eta|\ball|$.

$\psi_3(x+y)$ can be expressed as an affine function of $x$ plus an affine function of $y$;
these functions can be incorporated into $a_1(x)$, $a_2(y)$, respectively.
Combining this information with the hypotheses therefore gives
\begin{equation}
\big| a_1^\sharp(x)+a_2^\sharp(y) + \sigma(Lx,Ly) \big|\le CA
\end{equation}
for nearly all $(x,y)\in\ball\times\ball$, where $a_j^\sharp-a_j$ is affine.
Taking first differences with first to $x$ gives
\begin{equation}
\big| \Delta_h a_1^\sharp(x) + \sigma(Lh,Ly) \big|\le CA
\end{equation}
for nearly all $x,h,y\in\ball$ such that $x,h,x+h,y\in\ball$.
By specializing to a typical value of $y$, one finds that
there exists a function $h\mapsto c(h)$ such that $|\Delta_h a_1^\sharp(x)-c(h)|\le CA$
for nearly all $x,h\in\ball$ such that $x+h\in\ball$.
Therefore by Lemma~\ref{lemma:simplefunctleqn}
there exists an affine function $\psi$ such that $|a_1^\sharp-\psi|\le CA$
for nearly all points of $\ball$.
Since $a_1-a_1^\sharp$ is affine, the same conclusion holds for $a_1$.
Interchanging the roles of the variables $x,y$ in this argument produces the
same conclusion for $a_2$.

Combining these results for all $a_j$ with the original hypothesis, we conclude
that there exists an affine function $\psi$ of $(x,y)$ such that
$|\psi(x,y)-\sigma(Lx,Ly)|\le CA$ for nearly every $(x,y)\in\ball^2$.
The same must then hold for every $(x,y)\in\ball^*\times\ball^*$,
since $\psi,\sigma_L$ are polynomials.
By applying $\partial^2/\partial x_i\partial y_j$ for arbitrary indices $i,j$
and exploiting the affine character of $\psi$ together with the homogeneous quadratic
nature of $\sigma(Lx,Ly)$ we conclude that $|\sigma(Lx,Ly)|\le CA$ for all $(x,y)\in\ball^2$.
According to Lemma~\ref{lemma:matrixalgebra}, this implies the existence of $S\in\sp(2d)$ 
such that $\norm{SL^{-1}}\le CA^{1/2}$.
\end{proof}

\section{Proof of Theorem~\ref{thm:main} for nonnegative functions}

Let $\bp$ be an admissible ordered triple of exponents in $(1,\infty)^3$,
and let $\delta>0$ be small. 
Let $f_j\in L^{p_j}(\heis^d)$ for $j\in\{1,2,3\}$ satisfy $\norm{f_j}_{p_j}=1$,
as we may suppose without loss of generality. Set $\bff = (f_1,f_2,f_3)$. 
Assume that each $f_j\ge 0$, and suppose that 
\[\scriptt_{\heis^d}(\bff) \ge (1-\delta)\bestA_{\bp}^{2d+1}\norm{\bff}_\bp
= (1-\delta)\bestA_{\bp}^{2d+1}.\]
Let $\gamma=\gamma(\bp)=(\gamma_1,\gamma_2,\gamma_3)$ with $\gamma_1=1$.

Define $F_j:\reals^{2d}\to[0,\infty]$ 
and $f_{j,x}:\reals^1\to[0,\infty]$ as in \eqref{eq:Fjdefn}; however,
the definition of $f_{j,x}$ will be modified below, for those $x$ for which $f(x,t)$
vanishes for almost every $t$.
Set $\bF = (F_1,F_2,F_3)$.
For $\bx\in(\reals^{2d})^3$ define 
\begin{equation} f_{3,\bx}^\dagger(s) = f_3(x_3,s+\sigma(x_1,x_2));\end{equation}
as in \S\ref{section:recap}, this definition will only be relevant when $x_3=-x_1-x_2$.
Define a measure $\nu_\bF$ on $(\reals^{2d})^3$, supported on $\Lambda_{\reals^{2d}}$, by
\begin{equation}\label{nudefn} d\nu_\bF(\bx) = \prod_{j=1}^3 F_j(x_j)\,d\lambda_{\reals^{2d}}(\bx),\end{equation}
where $\lambda_{\reals^{2d}}$ is the natural $4d$--dimensional Lebesgue measure
on $\Lambda_{\reals^{2d}}$
introduced above.
Since $\norm{F_j}_{p_j}=\norm{f_j}_{p_j} =1$ and $\bp$ is admissible,
Young's inequality for $\reals^{2d}$ guarantees that 
%$\nu_\bF$ is a finite measure that satisfies
% whose total variation does not exceed 
$\nu_\bF(\reals^{2d}\times\reals^{2d}\times\reals^{2d}) \le \bestA_{\bp}^{2d}$.
% \prod_{j=1}^3\norm{F_j}_{p_j} = \bestA_{\bp}^{2d}\prod_{j=1}^3\norm{f_j}_{p_j}$.

\begin{lemma}\label{lemma:start} 
For each $d\ge 1$ and each admissible ordered triple $\bp$ there
exists $C<\infty$ with the following property.
Let $f_j\in L^{p_j}(\heis^d)$
be nonnegative and satisfy $\norm{f_j}_{p_j}=1$ for each $j\in\{1,2,3\}$.  
Let $\delta>0$. If 
$\scriptt_{\heis^d}(\bff)\ge (1-\delta)\bestA_{\bp}^{2d+1}$ then
\begin{equation} \scriptt_{\reals^{2d}}(\bF) \ge (1-\delta)\bestA_{\bp}^{2d}
% \prod_{j=1}^3\norm{F_j}_{p_j}
%= (1-\delta)\bestA_{\bp}^{2d}\norm{\bff}_\bp
\end{equation}
and there exists a set $E\subset\Lambda_{\reals^{2d}}$ 
satisfying
\begin{equation} \nu_\bF(E)\le C\delta^{1/2}\end{equation} 
such that for every 
$\bx\in\Lambda_{\reals^{2d}}\setminus E$,
\begin{equation} \left \{
\begin{aligned}
&F_j(x_j)\ne 0 \text{ for each $j\in\{1,2,3\}$,} 
\\
&\scriptt_{\reals^1}\big(f_{1,x_1},f_{2,x_2},f^\dagger_{3,\bx}\big)
\ge (1-o_\delta(1)) \bestA_{\bp}.
\label{eq:start2}
% \ \text{ for all $\bx=(x_1,x_2,x_3)\in\Lambda_{\reals^{2d}}\setminus E$.}
\end{aligned}
\right. \end{equation}
\end{lemma}

A proof of Lemma~\ref{lemma:start}
is implicit in the proof in \S\ref{section:recap} that the optimal constant
in Young's inequality for $\heis^d$ does not exceed the optimal constant for $\reals^{2d+1}$.
Details are left to the reader. \qed

According to Theorem~\ref{thm:youngest} there exists
an ordered triple $\bG = (G_1,G_2,G_3)$ of Gaussians $G_j:\reals^{2d}\to\complex$
that extremizes Young's convolution inequality for $\reals^{2d}$, of the form
\[G_j(x) = c_j |\det(L)|^{1/p_j} e^{-\gamma_j|L(x-a_j)|^2},\] 
where $\gamma=\gamma(\bp)$, $a_1 + a_2 + a_3=0$, $c_j>0$,
and $L$ is an invertible linear endomorphism of $\reals^{2d}$,
such that $\norm{F_j-G_j}_{L^{p_j}(\reals^{2d})}=o_\delta(1)$.
The constants $c_j$ are determined by requiring that $\norm{G_j}_{p_j}=1$,
as we may require with no loss of generality since $\norm{F_j}_{p_j}=1$.
Exponential factors $e^{ix\cdot b_j}$ appear in the conclusion
of Theorem~\ref{thm:youngest} but can dropped; since $F_j\ge 0$ by its definition,
$|G_j|$ is at least as accurate an approximation  to $F_j$ in $L^{p_j}$ norm as is $G_j$.

Define an ordered triple of diffeomorphisms of $\heis^d$ by
\[(\psi_1(z_1),\psi_2(z_2),\psi_3(z_3)) = (z_1u,u^{-1}z_2v,v^{-1}z_3)\]
where $u = (-a_1,0)$ and $v = (-a_1-a_2,0)$.
Then $v^{-1} = (a_1+a_2,0) = (-a_3,0)$.
The triple $\Psi = (\psi_j^*)_{1\le j\le 3}$ is an element of $\fg(\heis^d)$, so
upon replacement of $f_j$ by $f_j\circ\psi_j$
all of the assumptions and conclusions above are unaffected, and we
gain the simplification
\[G_j(x) = c_j |\det(L)|^{1/p_j} e^{-\gamma_j|Lx|^2}.\] 

\begin{lemma} \label{lemma:alongtheway}
Let $\bff,L,G_j$ be as above.
There exist $\lambda\in\reals^+$, $S\in\sp(2d)$, 
positive scalars $c_j$,
a set $E'\subset\Lambda_{\reals^{2d}}$,
affine mappings $\varphi_j:\reals^{2d}\to\reals^1$,
and Lebesgue measurable functions 
$h_j: \reals^{2d}\to[0,\infty)$ of the form
\begin{equation} h_{j}(x,t) = c_j e^{-\lambda\gamma_j(t-\varphi_j(x))^2}\end{equation}
such that $h_{j,x}(t) = h_j(x,t)$ satisfy the following conclusions:
\begin{align}
& \norm{h_{j,x}}_{L^{p_j}(\reals)}=1 \text{ for every $x\in\reals^{2d}$}
\label{eq:hjxnormal}
\\ & \nu_\bF(E') \le  o_\delta(1)
\\ & \norm{f_{j,x_j}-h_{j,x_j}}_{p_j}\le o_\delta(1)
\ \text{for each $j\in\{1,2,3\}$ for all $\bx\in\Lambda_{\reals^{2d}}\setminus E'$,}
\\&\norm{SL^{-1}}\le o_\delta(1)\lambda^{-1/4},
\\& \varphi_1(x_1)+\varphi_2(x_2)+\varphi_3(x_3)\equiv 0
\ \text{whenever $\bx\in\Lambda_{\reals^{2d}}$.}
\end{align}
\end{lemma}

Here $F_j$ is associated to $f_j$ as indicated above, and $(\gamma_1,\gamma_2,\gamma_3) =\gamma(\bp)$.
%$h_{j,x}$ is defined for every $x\in\reals^{2d}$, but its definition for those $x\in E'$
%is irrelevant, entering only in the normalization \eqref{eq:hjxnormal}.

\begin{proof}
Temporarily make the change of variables
$(x,s)\mapsto (y,t)$ in $\heis^d$, with
\begin{equation} \label{secondchange} y = L(x)  \ \text{ and } \ t=s. \end{equation} 
We make this same change of variables for each index $j\in\{1,2,3\}$.
The resulting diffeomorphism of $(\heis^d)^3$
corresponds to an element of $\fg(\heis^d)$ if and only if $L\in\sp(2d)$, which need not hold.
So we will revert to the original coordinates after exploiting these new coordinates. 

\begin{comment}
$((x_j,s_j))_{1\le j\le 3} \mapsto ((y_j,t_j))_{1\le j\le 3}$ in $\heis^d$
with 
\begin{equation} \label{secondchange} 
y_j = L(x_j)  \ \text{ and } \ t_j=s_j. \end{equation} 
Such a change of variables corresponds to an element of $\fg(\heis^d)$ if and only if $L\in\sp(2d)$,
which need not hold.
Therefore this change of variables is made only temporarily; we will revert to the coordinates below
after using the new coordinates to obtain certain information.
\end{comment}

Set \begin{equation}\tilde f_j(y_j,t)=f_j(L^{-1}y_j,t),\end{equation}
and of course $\tilde f_{j,y_j}(t) = \tilde f_j(y_j,t)$.
In these modified coordinates and for these modified functions,
the conclusions of Lemma~\ref{lemma:start}, coupled
with the approximations $\norm{F_j-G_j}_{p_j}=o_\delta(1)$, 
can be stated as follows. 
Set \begin{gather}
\tilde G_j(y) = c_j e^{-\gamma_j|y|^2}.
\\
d\nu_{\tilde\bG}(\by) = \prod_{j=1}^3 \tilde G_j(y) \,d\lambda_{\reals^{2d}}(\by)
\\
\tilde f^\dagger_{3,\by}(s) = \tilde f_{3,y_3}(s + \sigma_{L}(y_1,y_2))
\end{gather}
Recall the notation $\sigma_{L}(y_1,y_2) = \sigma(L^{-1}y_1,L^{-1}y_2)$. 
By Lemma~\ref{lemma:start}, since $\sum_{j=1}^3 p-j^{-1}=2$,
there is a set $E\subset\Lambda_{\reals^{2d}}$ satisfying $\nu_{\tilde\bG}(E)=o_\delta(1)$ such that
\begin{equation} \label{eq:start3}
\scriptt_{\reals^1}\big(\tilde f_{1,y_1}, \tilde f_{2,y_2},\tilde f^\dagger_{3,\by}\big)
\ge (1-o_\delta(1)) \bestA_{\bp}
\ \text{ for all $\by=(y_1,y_2,y_3)\in\Lambda_{\reals^{2d}}\setminus E$.}
\end{equation}
Moreover, $\norm{\tilde f_{j,y_j}}_{p_j}=1$ for each $j\in\{1,2,3\}$
whenever $\by\in\Lambda_{\reals^{2d}}\setminus E$.

Let $\delta\mapsto\rho(\delta)$
% \le\rho_0(\delta)$ 
be a function that tends to infinity slowly as $\delta\to 0^+$,
to be chosen below. This function may also depend on $d,\bp$ but is independent of $\bff$.
Define $\ball$ to be the closed ball of radius $\rho(\delta)$ centered at the origin in $\reals^{2d}$.
The $L^{p_j}$ norm of $\tilde G_j$ on the complement of $\ball$ is $o_\delta(1)$
since $\lim_{\delta\to 0} \rho(\delta)=\infty$.
$\tilde G_j$ is bounded above uniformly in $\delta$, and on $\ball^*$, 
and is bounded below by $ce^{-C\rho(\delta)^2}$.
Thus by \eqref{eq:start3}, under the convention that $\by = (y_1,y_2,y_3)$ is regarded
as a function $\by(y_1,y_2)$ of $(y_1,y_2)$ via the relation $y_3=-y_1-y_2$,
\eqref{eq:start3} holds
%\begin{equation} \scriptt_{\reals^1}(\tilde f_{1,y_1}, \tilde f_{2,y_2},
%\tilde f^\dagger_{3,\by})\ge (1-O(\delta^{1/2}))\bestA_{\bp} \end{equation}
for all $(y_1,y_2)\in\ball\times\ball$ outside a set of Lebesgue measure 
$\le \nu_{\tilde\bG}(E)c^{-1}e^{C\rho(\delta)^2}$.
% |\ball|^2$. 
Choose $\rho(\delta)$ to tend to infinity so slowly that this product is $\le o_\delta(1)|$
and hence, since $|\ball|\to\infty$ as $\rho\to\infty$, is $\le o_\delta(1)|\ball|^2$.
This is possible because $\nu_{\tilde\bG}(E)=o_\delta(1)$ tends to zero at a rate
that depends on $\delta,\bp,d$ but is otherwise independent of $\bff$ and of the choice of $\rho(\delta)$.
%The function $\rho(\delta)\le\rho_0(\delta)$ will be chosen later to satisfy an additional constraint.

By \eqref{eq:start3} and Theorem~\ref{thm:youngest}, for each $j\in\{1,2,3\}$,
for all $y_j\in\ball$ outside a set whose Lebesgue measure is $o_\delta(1)|\ball|$,
there exists a positive Gaussian function $\reals^1\owns t\mapsto g_{j,y_j}(t)$ satisfying
$\norm{\tilde f_{j,y_j}-g_{j,y_j}}_{p_j} \le o_\delta(1)$.
These functions can be chosen to depend Lebesgue measurably on the parameters $y_j$.

Write $g_{j,y}(t) = c_j(y)e^{-\lambda_j(y)(t-\alpha_j(y))^2}$
where $\lambda_j,c_j,\alpha_j$ are measurable functions with domains $\reals^{2d}$;
$\lambda_j,c_j$ take values in $(0,\infty)$ and $\alpha_j$ takes values in $\reals^1$. 
For all $y_j\in\ball$ outside a set of Lebesgue measure $\le o_\delta(1)|\ball|$, 
$\norm{\tilde f_{j,y_j}}_{p_j}=1$. 
Therefore $(g_{1,y_1},g_{2,y_2},g^\dagger_{3,-y_1-y_2})$
nearly extremizes Young's inequality for $\reals^1$, for all $(y_1,y_2)\in\ball^2$
outside a set of Lebesgue measure $\le o_\delta(1)|\ball|^2$.

A first consequence of this near extremality is that
\begin{equation}
\left| \frac{\lambda_i(y_i)}{\lambda_j(y_j)} - \frac{\gamma_i}{\gamma_j} \right|
 = o_\delta(1)
\end{equation}
for all $(y_1,y_2,y_3)\in \ball^3$ outside a set of Lebesgue measure $o_\delta(1)|\ball|^2$ 
for all indices $i,j\in\{1,2,3\}$, where $y_3$ continues to be defined to be $-y_1-y_2$.
Therefore there exists $\lambda\in\reals^+$ such that
\begin{equation} \lambda_j(y) = \lambda \cdot ( \gamma_j + o_\delta(1)) 
\ \text{ for each index $j\in\{1,2,3\}$,} \end{equation}
for all $ y\in\ball$ outside a set of Lebesgue measure $o_\delta(1)|\ball|$.
Thus for each $j\in\{1,2,3\}$,
\begin{equation} \label{eq:gform}
\big| g_{j,y}(t) \ -\  c'_j e^{-\lambda\gamma_j(t-\alpha_j(y))^2}\big| \le  o_\delta(1)
\end{equation}
in $L^{p_j}(\reals^1)$ norm, for every $y\in\ball$ outside a set of 
Lebesgue measure $o_\delta(1)|\ball|$.
The coefficients $c'_j$ are now constants, rather than functions of $y\in\reals^{2d}$.

In order for $(g_{1,y_1},g_{2,y_2},g^\dagger_{3,\by})$, 
with $g_{j,y_j}$ of the form \eqref{eq:gform} and $y_3= y_3(y_1,y_2) = -y_1-y_2$,
to $(1-o_\delta(1))$--nearly extremize Young's inequality for $\reals^1$ for every
$(y_1,y_2)\in \ball^2$ outside a set of Lebesgue measure $o_\delta(1)|\ball|^2$, it is necessary that
\begin{equation}
\alpha_1(y_1)+\alpha_2(y_2)+\alpha_3(-y_1-y_2) + \sigma_{L}(y_1,y_2) 
\le \lambda^{-1/2} \cdot o_\delta(1)
\end{equation}
for all $(y_1,y_2)\in\ball^2$ outside a set of Lebesgue measure $o_\delta(1)|\ball|^2$.
By Proposition~\ref{prop:applytoaj}, this implies the existence of 
affine functions $\varphi_j:\reals^{2d}\to\reals$ satisfying for each $j\in\{1,2,3\}$
\begin{equation}\label{eq:leadstocloseness} 
|\alpha_j(y)- \varphi_j(y)| \le o_\delta(1)\cdot\lambda^{-1/2}
\end{equation}
for all $y\in\ball$ outside a set of Lebesgue measure $o_\delta(1)|\ball|$, 
and satisfying 
\[ \varphi_1(x_1)+\varphi_2(x_2) + \varphi_3(-x_1-x_2)\equiv 0.  \]
Moreover, there exists $S\in\sp(2d)$ such that
\begin{equation} \norm{SL^{-1}}
\le \lambda^{-1/4}o_\delta(1)|\ball|^{-1/2d}
\le \lambda^{-1/4}o_\delta(1).\end{equation}
Equivalently, $L = \tilde L\circ S$ where $\tilde L$ satisfies a lower bound
\begin{equation}\label{relate} |\tilde L(v)| \ge \lambda^{1/4}\eta(\delta)^{-1}|v| \end{equation}
uniformly for all $0\ne v\in\reals^{2d}$, where $\eta(\delta)\to 0$ as $\delta\to 0$.
These properties of $L$ will be exploited below.

Define Gaussian functions
\begin{equation}
\tilde g_{j,y}(t) =  c'_j e^{-\lambda\gamma_j(t-\varphi_j(y))^2}.
\end{equation}
\eqref{eq:leadstocloseness}
implies that $\norm{\tilde f_{j,y_j}-\tilde g_{j,y_j}}_{p_j}\le o_\delta(1)
= o_\delta(1)\norm{\tilde f_{j,y_j}}_{p_j}$ for all $\by\in\Lambda_{\reals^{2d}}\setminus E'$,
with $\nu_\bF(E')= o_\delta(1)$.
A consequence, since $\tilde G_j\in L^1$, is that
\begin{equation}
\norm{\tilde f_{j,y}(t)F_j(y) - \tilde g_{j,y}(t)\tilde G_j(y)}_{L^{p_j}(\ball\times\reals,\,dy\,dt)}
\le o_\delta(1) \end{equation}
for each $j\in\{1,2,3\}$.
Therefore
\begin{equation}
\norm{\tilde f_{j,y}(t)F_j(y) - \tilde g_{j,y}(t)\tilde G_j(y)}_{L^{p_j}(\reals^{2d}\times\reals,\,dy\,dt)}
\le o_\delta(1). \end{equation}

Returning to the original coordinates $(x,t)$ for $\heis^d$, define
\begin{equation}
\tilde h_j(x,t) = \tilde g_{j,y}(t) = \tilde g_{j,L(x)}(t) = c'_j e^{-\lambda\gamma_j(t-\varphi_j\circ L(x))^2}.
\end{equation}
The next step is to simplify matters by exploiting symmetries. 
We apply in sequence two elements $\Psi\in\fg(\heis^d)$. 
The first is $\Psi=(\psi_1,\psi_2,\psi_3)$, with
$\psi_j(x_j,t_j) = t_j-\varphi_j\circ L(x_j)$. 
The second takes the form $\psi_j(x,t) = (S(x),t)$, 
where $S$ is as in \eqref{relate}.
Replace $f_j$ by $f_j\circ\psi_j$
for each of these in turn, continuing to denote by $f_j$ the resulting functions
and by $F_j$ the associated functions with domains $\reals^{2d}$.
Likewise compose $\tilde h_j$ with each of these in turn, and denote
by $h_j^\sharp$ the resulting composed functions.
Matters are thereby reduced to the situation in which 
\begin{align*} 
& h_{j,x}^\sharp(t) = c_j e^{-\lambda\gamma_j  t^2},
\\&
F_j(x) = c_j e^{-\gamma_j |\tilde L(x)|^2},
\\&
\norm{f_{j,x_j}-h_{j,x_j}^\sharp}_{p_j} \le o_\delta(1)
\ \forall\, \bx\in\Lambda_{\reals^{2d}}\setminus E''
%  is independent of $x$, and
\end{align*}
where $E''\subset\Lambda_{\reals^{2d}}$ satisfies $\nu_{\bF}(E'')\le o_\delta(1)$
and $\tilde L,\lambda$ are related by \eqref{relate}.

% ?? Is the following needed?
%The action of $\sp(2d)$ on $\reals^{2d}$ preserves Lebesgue measure and 
%maps any ball centered at the origin to itself, 
%so $\ball$ continues to play the same role in these new coordinates.

The next reduction is
an automorphic change of variables in $\heis^d$ of the form
\[(x,t)\mapsto \psi(x,t) =  (z,r) = (\eta(\delta)^{-1}\lambda^{1/4}x,\eta(\delta)^{-2}\lambda^{1/2}t),\]
where $\eta(\delta)$ is the function introduced in \eqref{relate}.
Setting $\psi_j=\psi$ for all three indices $j$ defines an element $\Psi\in\fg(\heis^d)$.
%Invoke the factorizations $f_j(x,t) = F_j(x)\tilde f_{j,x}(t)$
%together with the approximations for $F_j,\tilde f_{j,x}$ derived above.
In these new coordinates, the conclusion is that
$\norm{f_j-f_j^*}_{p_j} \le o_\delta(1)$ where
\[f_j^*(z,r) = c_j e^{-\gamma_j |L'z|^2} e^{-\gamma_j\eps r^2},\]
where $L':\reals^{2d}\to\reals^{2d}$ 
is linear and satisfies $|L'z|\ge|z|$ for all $z\in\reals^{2d}$,
and $\eps\le \eps(\delta)$ where $\eps(\delta)$ tends to $0$  as $\delta\to 0$,
and depends also on $\bp,d$ as well as on $\delta$, but not otherwise on $\bff$.
This completes the analysis of nonnegative near-extremizers $\bff$.
\end{proof}

\section{The complex-valued case} \label{section:complex}
Let $\delta>0$ be small, and 
consider an arbitrary complex-valued $\bff = (f_1,f_2,f_3)$
satisfying $\norm{f_j}_{p_j}\ne 0$ for each index $j$,
and $|\scriptt_{\heis^d}(\bff)| \ge (1-\delta)\bestA_\bp^{2d+1}\norm{\bff}_\bp$.
Since $\scriptt_{\heis^d}(|f_1|,|f_2|,|f_3|) \ge |\scriptt_{\heis^d}(\bff)|$,
we may apply the result proved above for nonnegative near-extremizers
to conclude that there exists $\Psi=(\psi_1,\psi_2,\psi_3)\in\fg(\heis^d)$
such that for each $j\in\{1,2,3\}$,
\[ \left\{ \begin{aligned}
&\norm{|f_j\circ\psi_j|-G_j}_{p_j} \le o_\delta(1)\norm{f_j\circ\psi_j}_{p_j}
\\& G_j(x,t) = c_j e^{-\gamma_j |Lx|^2}e^{-\gamma_j \eps t^2}
\end{aligned} \right.\]
where $c_j\asymp 1$, $|Lx|\ge |x|$ for all $x\in\reals^{2d}$,
and $\eps\le o_\delta(1)$.
By replacing $f_j$ by $f_j\circ\psi_j$ multiplied by an appropriate normalizing constant
factor, we may also assume that $\norm{f_j}_{p_j}=1$
and then likewise that $\norm{G_j}_{p_j}=1$.

Write $f_j = e^{i\alpha_j}|f_j|$ where $\alpha_j: \heis^d\to\reals$
is measurable.
We seek to analyze the factors $e^{i\alpha_j}$.
Since $\norm{f_j - e^{i\alpha_j}G_j}_{p_j}
= \norm{|f_j| - G_j}_{p_j} \le o_\delta(1)$,
\[|\scriptt_{\heis^d}(e^{i\alpha_1}G_1,e^{i\alpha_2}G_2,e^{i\alpha_3}G_3)|
\ge (1-o_\delta(1))\bestA_p^{2d+1}.\]
Thus it suffices to prove that $(e^{i\alpha_j}G_j: 1\le j\le 3)$
satisfies the conclusions of Theorem~\ref{thm:main}.
So we redefine $f_j$ to be  $e^{i\alpha_j}G_j$ henceforth. 

By multiplying these functions by unimodular constants, we may assume without loss
of generality that $\scriptt_{\heis^d}(\bff)$ is real and positive.
Since then $\Re\scriptt_{\heis^d}(f_1,f_2,f_3) \ge (1-o_\delta(1))\scriptt_{\heis^d}(|f_1|,|f_2|,|f_3|)$,
\begin{equation}\label{eq:multclose}
|\prod_{j=1}^3 e^{i\alpha_j(z_j)}-1|=o_\delta(1)
\ \text{ for all $\bz\in(\heis^d)^3$ outside a set 
satisfying $\nu_\bG(E)\le o_\delta(1)$} \end{equation}
where
$ d\nu_\bG(\bz) = \prod_j G_j(z_j)\,d\lambda_{\heis^d}(\bz)$.

Let $\rho=\rho(\delta)$ be a positive quantity that tends to infinity 
slowly as $\delta\to 0$ and is to be chosen below, and let $\ball\subset\reals^{2d}$ be the ball of radius $1$
centered at $0$.
By \eqref{eq:multclose},
%$y\in\ball^3$ satisfying $\sum_{j=1}^3 y_j=0$, 
%where $\ball\subset\reals^{2d}$ is a ball whose radius tends to infinity sufficiently slowly
%as $\delta\to 0$,
\begin{equation}\label{mult:tosub} \Big| e^{i\alpha_1(L^{-1}y_1,t_1)} e^{i\alpha_2(L^{-1}y_2,t_2)}
e^{i\alpha_3 (-L^{-1}y_1-L^{-1}y_2,-t_1-t_2 - \sigma_L(y_1,y_2))}
-1 \Big| \le o_\delta(1)\end{equation}
for all
$((y_1,t_1),(y_2,t_2))\in (\ball\times[-\rho \eps^{-1/2},\rho \eps^{-1/2}])^2$
outside a set of Lebesgue measure less than or equal to $o_\delta(1)\cdot\eps^{-1}$
provided that the function $\rho$ is chosen so that
$\rho(\delta)\to\infty$ sufficiently slowly as $\delta\to 0$.
Therefore according to Lemma~\ref{lemma:nearlycharacter},
for each index $j$, $\alpha_j$ takes the form
\begin{equation}\label{eq:amultreln}
e^{i\alpha_j(L^{-1}y,t)} = e^{i (a_j(y)t+b_j(y)+o_\delta(1))}\end{equation}
for $y\in\ball$ and $|t|\le \rho(\delta) \eps^{-1/2}$
outside a set of Lebesgue measure $o_\delta(1)\eps^{-1/2}$.
The coefficients $a_j,b_j$ are real-valued measurable functions.

We will use informal language ``for nearly all $y\in\ball$''
to indicate a Lebesgue measurable subset $A\subset \ball$
satisfying $|A|\le o_\delta(1)|\ball|$, 
where the quantity $o_\delta(1)$ depends on $\delta,\bp,d$ alone
and tends to $0$ as $\delta\to 0$ while $\bp,d$ remain fixed.
``Nearly all $(y_1,y_2)\in\ball^2$'' has a corresponding meaning.

Invoking \eqref{eq:amultreln} together with \eqref{mult:tosub} 
for typical $(t_1,t_2)$ and also for typical $(t'_1,t'_2)$ satisfying
$|t_j|,|t'_j|\le \rho(\delta)\eps^{-1/2}$, considering products
of the exponential factors, and setting $u_j = t'_j-t_j$ gives
\begin{equation} \label{eq:givesn}
\Big| e^{i u_1 a_1(L^{-1}y_1)} e^{i u_2 a_2(L^{-1}y_2)}
e^{-i (u_1+u_2) a_3 (-L^{-1}y_1-L^{-1}y_2)}
-1 \Big| \le o_\delta(1)\end{equation}
for nearly all $(y_1,y_2)\in\ball^2$ and nearly all $(u_1,u_2)$ satisfying 
$|u_j|\le \tfrac12 \rho(\delta)\eps^{-1/2}$
outside a set of Lebesgue measure $o_\delta(1)\eps^{-1}$.
The advantage of \eqref{eq:givesn} over \eqref{mult:tosub}
is that $b_j$ and $\sigma_L$ have been eliminated.

This last inequality can be equivalently written
\begin{equation} \label{eq:beyondstarted}
\Big| 
e^{i u_1 [a_1(L^{-1}y_1)-a_3(-L^{-1}(y_1-y_2))]} 
e^{i u_2 [a_2(L^{-1}y_1)-a_3(-L^{-1}(y_1-y_2))]} 
-1 \Big| \le o_\delta(1).\end{equation}
By Lemma~\ref{lemma:lastgasp}, below, \eqref{eq:beyondstarted} implies that
\begin{equation}
\big|a_1(L^{-1}(y_1)-a_3(-L^{-1}(y_1-y_2))\big| \le o_\delta(1)\eps^{1/2} 
\end{equation}
for nearly all $(y_1,y_2)\in\ball^2$.
Note that unlike the functions $\alpha_j$, which are only determined up to addition
of arbitrary measurable functions taking values in $2\pi\integers$,
the constituent parts $a_j$ can be pinned down as $\reals$--valued,
rather than $\reals/2\pi\integers$--valued, functions.

Therefore there exists a real number $\tilde a$
such that $|a_j(L^{-1}y)-\tilde a|\le o_\delta(1)\eps^{1/2}$
for nearly all $y\in\ball$ for $j=1,3$.
The same reasoning gives the same conclusion for $j=2$.
Thus for each $j\in\{1,2,3\}$,
\begin{equation}
e^{i\alpha_j(L^{-1}y,t)} = e^{i \tilde a t} e^{i b_j(y)} + o_\delta(1)
\end{equation}
for all $(y,t)\in\ball\times[-\rho(\delta)\eps^{-1/2},\rho(\delta)\eps^{-1/2}]$
outside a set of Lebesgue measure $o_\delta(1)\eps^{-1/2}$.
Thus 
\begin{equation} \norm{ e^{i\alpha_j(x,t)}G_j(x,t) 
- e^{i(\tilde a t + b_j(L_j(x))}G_j(x,t)}_{L^{p_j}(\heis^d)}
\le o_\delta(1)\norm{f_j}_{p_j};\end{equation}
so we may replace $\alpha_j(x,t)$ by $\tilde a t + b_j(L(x))$.

Inserting this into \eqref{mult:tosub} gives
\begin{equation} \label{gettinglonger}
\Big| e^{i b_1(L^{-1}y_1)} e^{i b_2(L^{-1}y_2)}
e^{i b_3 (-L^{-1}y_1-L^{-1}y_2)}
e^{-i\tilde a \sigma_L(y_1,y_2)}
-1 \Big| \le o_\delta(1)\end{equation}
for nearly all $(y_1,y_2)\in\ball^2$.
From the antisymmetry of $\sigma_L$ it follows that
\begin{equation}
\Big| e^{i (b_1+b_2)(L^{-1}y_1)} e^{i (b_1+b_2)(L^{-1}y_2)}
e^{i 2b_3 (-L^{-1}y_1-L^{-1}y_2)}
-1 \Big| \le o_\delta(1)\end{equation}
for nearly all $(y_1,y_2)\in\ball^2$;
this can be deduced by interchanging
$y_1$ with $y_2$ and considering the product of the two
resulting left-hand sides of \eqref{gettinglonger}. 

According to Lemma~\ref{lemma:nearlycharacter},
the functions $e^{i2b_3\circ L^{-1}}$ and $e^{i(b_1+b_2)\circ L^{-1}}$
nearly agree with exponentials of imaginary affine functions,
at nearly all points of $\ball$.
Since
\begin{equation}
\Big| e^{i 2b_1(L^{-1}y_1)} e^{i 2b_2(L^{-1}y_2)}
e^{i 2b_3 (-L^{-1}y_1-L^{-1}y_2)}
e^{-i2\tilde a \sigma_L(y_1,y_2)}
-1 \Big| \le o_\delta(1)\end{equation}
for nearly all $(y_1,y_2)\in\ball^2$ by \eqref{gettinglonger},
it follows by invoking this information for $b_3$ that
\[ e^{i 2b_1(L^{-1}y_1)} e^{i 2b_2(L^{-1}y_2)}
e^{-i2\tilde a \sigma_L(y_1,y_2)}\]
is nearly equal to the exponential of an imaginary affine function of $(y_1,y_2)$,
at nearly all points of $\ball^2$.

Next consider the ratio
\begin{equation} \label{ratioconsidered}
\frac  {e^{i 2b_1(L^{-1}y_1)} e^{i 2b_2(L^{-1}(u+y_2))} e^{-i2\tilde a \sigma_L(y_1,u+y_2)}}
{e^{i 2b_1(L^{-1}y_1)} e^{i 2b_2(L^{-1}y_2)} e^{-i2\tilde a \sigma_L(y_1,y_2)}}
= e^{i 2b_2(L^{-1}(u+y_2))} e^{-i2b_2(L^{-1}(y_2))} e^{-i2\tilde a \sigma_L(y_1,u)}. \end{equation}
From the conclusion of the preceding paragraph one can deduce that
the right-hand side of \eqref{ratioconsidered}
% \[ e^{i 2b_2(L^{-1}(u+y_2))} e^{-i2b_2(L^{-1}(y_2))} e^{-i2\tilde a \sigma_L(y_1,u)}\]
nearly coincides with the exponential of an imaginary affine function of $u$ alone,
at nearly all points $(y_1,y_2,u)$ with $y_1\in\ball$ and $y_2,u\in \tfrac12\ball$.
On the right-hand side, only the last exponential factor depends on $y_1$, so
by regarding this quantity as a function of $y_1$
we conclude that $|\tilde a|\cdot|\sigma_L(v,u)|\le o_\delta(1)$
for nearly all $(v,u)\in (\tfrac14 \ball)^2$. 
%Since $\sigma$ is a nondegenerate quadratic form, it follows that $\tilde a$ satisfies the upper bound
Therefore
\begin{equation}\label{tildeaLbound}
|\tilde a| \cdot \sup_{|x|,|y|\le 1} |\sigma_L(x,y)| \le o_\delta(1). \end{equation}
%Equivalently, \begin{equation} \label{tildeabound} 
%|\tilde a| \cdot \norm{(L^{-1})^*J L^{-1}} \le o_\delta(1).  \end{equation}
Therefore by Lemma~\ref{lemma:matrixalgebra}, below,
there exists $S\in\sp(2d)$ such that 
$|\tilde a|\cdot \norm{S L^{-1}}^2 \le o_\delta(1)$.

Combining this with \eqref{gettinglonger} yields
\begin{equation} \label{gettingshorter}
\Big| e^{i b_1(L^{-1}y_1)} e^{i b_2(L^{-1}y_2)}
e^{i b_3 (-L^{-1}y_1-L^{-1}y_2)}
-1 \Big| \le o_\delta(1)\end{equation}
for nearly all $(y_1,y_2)\in\ball^2$.
By Lemma~\ref{lemma:nearlycharacter}, for each $j\in\{1,2,3\}$ there exists an affine
function $L_j:\reals^{2d}\to\reals$ such that
\[ |e^{ib_j(L^{-1}y)}-e^{iL_j(y)}|\le o_\delta(1)\]
for nearly all $y\in\ball$.
Thus
\begin{equation} e^{i\alpha_j(L^{-1}y,t)} = e^{i\tilde a t}e^{iL_j(y)} + o_\delta(1) \end{equation}
for $(y,t)\in\ball\times\reals$ satisfying $|t|\le \rho(\delta)\eps^{-1/2}$
outside a set of Lebesgue measure $\le o_\delta(1)\eps^{-1/2}$,
where $\tilde a$ satisfies \eqref{tildeaLbound}.

This concludes the proof of Theorem~\ref{thm:main} in the general complex-valued case.
\qed

\section{Some matrix algebra}
\begin{lemma}\label{lemma:matrixalgebra}
For any invertible linear endomorphism $L:\reals^{2d}\to\reals^{2d}$,
\begin{equation}
\norm{L^*JL}^{1/2} = \inf_{S\in\sp(2d)} \norm{S^{-1}L}.
\end{equation}
\end{lemma}

%\begin{lemma}\label{lemma:matrixalgebra_older}
%Let $A\in[1,\infty)$.
%If $\norm{L^*JL}\le A$ then there exists $S\in\sp(d)$ such that $\norm{S^{-1}L}\le A^{1/2}$.
%\end{lemma}

\begin{proof}
That $\norm{L^*JL}\le \inf_{S\in\sp(2d)} \norm{S^{-1}L}^2$
is immediate. For any $L$ and any $S\in\sp(2d)$,
\begin{align*}
\norm{L^*JL}
= \norm{(S^{-1}L)^*S^* J S(S^{-1}L)}
= \norm{(S^{-1}L)^* J (S^{-1}L)}
\le \norm{S^{-1}L}\norm{J}\norm{S^{-1}L}
= \norm{S^{-1}L}^2.
\end{align*}

To establish the reverse inequality, note that since
$L^*JL$ is a nonsingular antisymmetric real matrix, its eigenvalues are imaginary,
and come in conjugate pairs; if $i\lambda$ is an eigenvalue then 
$\lambda\ne 0$ and
$-i\lambda$ is also an eigenvalue, and the eigenspace associated to $-i\lambda$
has the same dimension as the eigenspace associated to $i\lambda$; 
coordinatewise complex conjugation interchanges these two eigenspaces.
Therefore $L^*JL$ can be written in the form $\scripto_1^* K\scripto_1$
where $\scripto_1\in O(2d)$ and $K$ takes the form
\begin{equation}
K = 
\begin{pmatrix}
0 & t_1 & 0 & 0 & \cdots &0
\\
-t_1 & 0 & 0  & 0 & \cdots &0
\\
0 & 0 & 0 & t_2 &    \cdots &0
\\
0 & 0 & -t_2 & 0 &  \cdots &0
\\
\vdots& \vdots & \vdots & \vdots & \cdots & \vdots
\end{pmatrix}
\end{equation}
with $2\times 2$ blocks $\begin{pmatrix} 0 & t_j \\ -t_j & 0\end{pmatrix}$
along the diagonal, where $t_j\in\reals^+$ and  the eigenvalues 
are $\pm i t_j$.
Now $t_j \le \norm{L^* J L}$.
Defining 
\begin{equation} T = 
\begin{pmatrix}
t_1^{1/2} & 0  & 0 & 0 & 0 
&\cdots
\\
0 & t_1^{1/2} & 0 & 0 & 0
&\cdots
\\
0 & 0 & t_2^{1/2} & 0  & 0
&\cdots
\\
0 & 0 & 0 & t_2^{1/2} & 0 
&\cdots
\\
\vdots & \vdots & \vdots & \vdots & \vdots 
& \cdots
\end{pmatrix}
\end{equation}
gives
\begin{equation}
K = T^* \tilde J T
\end{equation}
where 
\begin{equation}
\tilde J = 
\begin{pmatrix}
0 & 1 & 0 & 0 & \cdots &0
\\
-1 & 0 & 0  & 0 & \cdots &0
\\
0 & 0 & 0 & 1 &    \cdots &0
\\
0 & 0 & -1 & 0 &  \cdots &0
\\
\vdots& \vdots & \vdots & \vdots & \cdots & \vdots
\end{pmatrix}
\end{equation}
with $2\times 2$ blocks $\begin{pmatrix} 0 & 1 \\ -1 & 0 \end{pmatrix}$
along the diagonal.
Now $\tilde J = \scripto_2^* J\scripto_2$ for an appropriate permutation matrix
$\scripto_2\in O(2d)$
and thus we have
\begin{equation} L^*JL = M^* J M \end{equation}
where $M = \scripto_2 T \scripto_1$.
Equivalently,
\begin{equation}
(LM^{-1})^* J (LM^{-1})=J,
\end{equation}
so $LM^{-1}\in \sp(2d)$. 
That is, $L= SM$ where $S\in\sp(2d)$.
Equivalently, $M=S^{-1}L$ satisfies 
\[\norm{M} =\norm{\scripto_2 T \scripto_1} \le\norm{\scripto_2}\norm{T}\norm{\scripto_1}
= \norm{T} = \norm{L^*JL}^{1/2}, \]
as required.
\end{proof}

\section{Integration of difference relations} \label{section:differencerelations}

In this section we establish Theorem~\ref{thm:polyfunctleqn}, which is motivated by 
considerations that have arisen in this paper, but on which the main theorems
do not rely. This is done in the hope that it will prove useful in other problems.
We continue to use the expressions ``nearly every'' and ``nearly all points'' in the same sense as 
in \S\ref{section:complex}.

The next lemma is elementary; the proof is omitted.
\begin{lemma} \label{lemma:extractcoeff}
Let $d,m\in\naturals$.
Let $q(x,y) = \sum_{0\le|\alpha|\le m} a_\alpha(y) x^\alpha$
where $a_\alpha$ are Lebesgue measurable functions.
Suppose that $|q(x,y)|\le 1$ for nearly every $(x,y)\in\ball\times\tilde\ball$.
Then for any multi-index $\beta$ satisfying $0\le|\beta|\le m$, 
$|a_\beta(y)|\le C$ where $C<\infty$ depends only on $m,d$.
\end{lemma}

Before embarking on the core of the proof of Theorem~\ref{thm:polyfunctleqn}
we introduce several simplifications.
Firstly, it suffices to prove this in the case in which $\ball$ is centered at $0$,
for the hypotheses and conclusions are invariant under translation. 
Second, it suffices to prove this for the ball $\ball$ centered at $0$
of radius $1$. For if the result holds for some ball centered at $0$,
then it holds uniformly for all such balls, because the hypotheses and conclusions
are invariant under dilations.  Thirdly,
it suffices to prove the theorem for $A=1$,
since hypotheses and conclusions are invariant under multiplication of $\varphi$
by positive scalars,
and the case $A=0$ follows from the case $A>0$ with uniform bounds
by a straightforward limiting argument.
Fourthly, assuming $\ball$ to be centered at the origin, 
it suffices to prove that there exists $\rho>0$, depending only on $d,D$,
such that the conclusion holds for all $x\in\rho\ball=\{\rho y: y\in\ball\}$
outside a set of measure $\eps \rho^d|\ball|$.
Indeed, the full conclusion for $\ball$ itself then follows by combining this weaker
conclusion with a Whitney decomposition of $\ball$, as in
\cite{christyoungest}. One arranges that each Whitney cube $Q_k$ is contained
in a ball $B_k$ of comparable diameter, such that the ball $B_k^*$ concentric
with $B_k$ with radius enlarged by a factor of $\rho^{-1}$ is contained in $\ball$.
Invoking the weaker result in its translation and dilation invariant form
gives an approximation by an affine function on $B_k$, provided that $|B_k|/|\ball|$
is not too small as a function of $\delta$. These affine functions patch together
on most of $\ball$ to yield a single globally defined affine function, up to 
a suitably small additive error.
The same reasoning reduces the case of small parameters $\eta$ to $\eta=1$.

The proof of the theorem will involve multiple steps in which $\ball$ is replaced
by a ball $\rho'\ball$ where $\rho'>0$ depends only on $d,D$. The final constant
$\rho$ is the product of all these factors $\rho'$. We will simplify notation
by allowing the value of $\rho$ to change from one step to the next, 
so that each of these factors $\rho'$, and products of successive factors,
are denoted by $\rho$.

The fifth simplification is one of language. Various conclusions will
hold for all $x\in\rho\ball$ except for a set of measure at most $\tau\rho^d|\ball|$
where $\tau>0$ depends only on $d,D,\delta$ and $\tau\to 0$ as $\delta\to 0$.
In this circumstance we will not specify a function $\delta\mapsto \tau(\delta)$,
but will simply write that the conclusions in question hold for nearly all $x\in\rho\ball$.
In the same sense we will write ``for nearly all $(x,y)\in\rho\ball\times\rho\ball$'',
and so on.

In the proof we write $O(1)$ for a quantity that is bounded above by some
constant depending only on $D,\eta$. The value of this quantity is permitted
to change from one occurrence to the next. 

We will argue by induction on the degree $D$.
The key to this induction is the observation that Theorem~\ref{thm:polyfunctleqn} implies an additional
conclusion.

\begin{corollary}\label{cor:additional}
Let $D$ be a nonnegative integer.
Under the hypotheses of Theorem~\ref{thm:polyfunctleqn},
for each multi-index satisfying $|\alpha|=D$,
there exists an affine function $\xi_\alpha$
such that 
the coefficients $a_\alpha$ in \eqref{eq:polyfun2} satisfy
\begin{equation} |a_\alpha(h)-\xi_\alpha(h)|\le CA\ \text{ for nearly all $h\in \rho\ball$.}\end{equation}
\end{corollary}

\begin{proof}
To prove this, assuming Theorem~\ref{thm:main} for the given degree $D$,
let $Q$ be a polynomial of degree $\le D+1$ that satisfies
the conclusion \eqref{eq:polyfunconclusion}.
Then assuming as we may that $\ball$ is centered at $0$ and has radius $1$,
$|\Delta_h Q(x)-\Delta_h \varphi(x)| \le CA$ for nearly all $(x,h)\in(\rho\ball)^2$.
Expand $\Delta_h Q(x) = \sum_{|\alpha|\le D} \tilde a_\alpha(h) x^\alpha$
where $\tilde a_\alpha$ are polynomials of degrees $\le D+1-|\alpha|$.
In particular, $\tilde a_\alpha$ is affine when $|\alpha|=D$.

Consider $\Delta_h Q-\Delta_h\varphi$. Substituting for $\Delta_h\varphi$ 
the expression $\sum_{|\alpha|\le D} a_\alpha(h)x^\alpha + O(A)$ given
in the hypothesis yields
\[ \big|\sum_{|\alpha|\le D} (a_\alpha(h)-\tilde a_\alpha(h))x^\alpha \big| \le CA\]
for nearly all $(x,h)\in (\rho\ball)^2$.
Invoking Lemma~\ref{lemma:extractcoeff}
gives $|a_\alpha(x)-\tilde a_\alpha(x)|\le CA$ for nearly all $x\in\rho\ball$,
which is the desired additional conclusion for $|\alpha|=D$.
\end{proof}

\begin{proof}[Proof of Theorem~\ref{thm:polyfunctleqn}]
We proceed by induction on $D$.  
Since the proof of
Corollary~\ref{cor:additional} for degree $D$ relied on Theorem~\ref{thm:main} for that same
degree, in the induction it is only permissible to
invoke Corollary~\ref{cor:additional} for smaller degrees.

The base case $D=0$ is a corollary of
Lemma~\ref{lemma:simplefunctleqn}.
Indeed, it is given that $|\varphi(x+h)-\varphi(x) - p(h)|\le A$
for nearly all points $(x,h)$ with $x\in\ball$ and $h\in\tilde\ball$,
where $p(h)$ is a polynomial of degree zero in $x$ that depends on $h$;
that is, $p(h)$ depends only on $h$.
If $\tilde\ball$ were equal to $\ball^*$
then this would be a direct application of Lemma~\ref{lemma:simplefunctleqn}.
The general case is proved by combining this special case with a Whitney decomposition 
of $\ball$, as in the analysis in \cite{christyoungest}.

In the proof for the inductive step, we operate under the following convention:
For $|\alpha|\le D-2$, $b_\alpha,\tilde b_\alpha,c_\alpha$ denote Lebesgue measurable functions,
with appropriate domains. An equation involving such functions
is to be interpreted as an existence statement; the assertion is that there exist measurable
functions such that the equation holds in the indicated domain.
These are permitted to change from one occurrence of each symbol to the next. 
However, this convention is not in force for $|\alpha|\ge D-1$; for such indices, the
functions $b_\alpha$ do not change after they are first introduced.

Assume without loss of generality that $A=1$.
For the inductive step, let $D\ge 1$, and let $\varphi,P$ satisfy the hypothesis
with $A=1$.  
For $x,s,t\in\rho\ball$ consider
\begin{align*}
\Delta_s\Delta_t\varphi(x)
&= \Delta_t\Delta_s\varphi(x)
\\&= \sum_{|\alpha|\le D} a_\alpha(s)((x+t)^\alpha-x^\alpha) + O(1)
\\&= \sum_{|\alpha|=D-1} \big( b_\alpha(s)\cdot t + b^\sharp_\alpha(s)\big) x^\alpha
+ \sum_{|\alpha|\le D-2} b_\alpha(s,t)x^\alpha + O(1)
\end{align*}
for nearly all $(x,s,t)\in (\rho\ball)^3$
where 
$s\mapsto b_\alpha(s)$ are $\reals^d$--valued measurable functions,
and $s\mapsto b^\sharp_\alpha(s)$ is real-valued and measurable.

The terms $b^\sharp_\alpha(s)$ are bothersome, because differences ought to vanish when $t=0$.
They can be eliminated by introducing an extra parameter
$t'\in \rho\ball$ and considering the resulting approximate functional equation
\begin{equation}\label{eq:stillbothersome}
\Delta_s\big( \Delta_t\varphi(x)- \Delta_{t'}\varphi(x)\big)
= \sum_{|\alpha|=D-1} b_\alpha(s)\cdot (t-t')x^\alpha  
+ \sum_{|\alpha|\le D-2} \tilde b_\alpha(s,t,t')x^\alpha + O(1),
\end{equation}
which holds for nearly all $(x,s,t,t')\in (\rho\ball)^4$.
Now 
\[\Delta_t\varphi(x)- \Delta_{t'}\varphi(x)
= \varphi(x+t)-\varphi(x+t') = \Delta_{t-t'}\varphi(x+t').\]
Therefore substituting $x=y-t'$ and then $\tau = t-t'$,
and specializing \eqref{eq:stillbothersome} to a typical value of $t'$,
gives
\begin{equation}\label{deltasdeltat}
\Delta_s\Delta_\tau \varphi(y)
= \sum_{|\alpha|=D-1} b_\alpha(s)\cdot\tau y^\alpha  + \sum_{|\alpha|\le D-2} c_\alpha(s,\tau)y^\alpha + O(1)
\end{equation}
for nearly all $(y,s,\tau)\in (\rho\ball)^3$,
where the coefficients $c_\alpha$ are measurable functions.

Specialize to a typical $\tau\in\rho\ball$.
With $\psi = \Delta_\tau\varphi$, this conclusion becomes
\begin{equation*}
\Delta_s \psi(y)
= \sum_{|\alpha|=D-1} b_\alpha(s)\cdot \tau y^\alpha + \sum_{|\alpha|\le D-2} c_\alpha(s,\tau)y^\alpha + O(1)
\end{equation*}
for nearly all $(y,s,\tau)\in (\rho\ball)^3$.
Therefore by induction on the degree $D$ and Corollary~\ref{cor:additional},
for each multi-index of degree $|\alpha|=D-1$,
there exists an $\reals^d$--valued affine function that agrees to within $O(1)$ at nearly
every point of $\rho\ball$ with $b_\alpha$. That is,
there exist $\tilde u_\alpha\in\reals^d\otimes\reals^d$
and $\tilde v_\alpha\in\reals^d$ such that
\begin{equation} \label{balphanearlyaffine}
|b_\alpha(s)-(\tilde u_\alpha\cdot s+\tilde v_\alpha)|=O(1)
\ \text{for nearly all $s\in\rho\ball$.}
\end{equation}

For $|\alpha|=D-1$, these coefficients $b_\alpha$ are related to the coefficients $a_\alpha$
in the hypothesis \eqref{eq:polyfun2} as follows:
Writing $b_\alpha(s) = (b_{\alpha,1}(s),\dots,b_{\alpha,d}(s))$,
letting $e_i\in\reals^d$ be the coordinate vector with $i$--th coordinate equal to $1$
and all other coordinates equal to $0$,
and writing $\alpha= (\alpha_1,\dots,\alpha_d)$, 
one has
\[b_{\alpha,i}(s) = (\alpha_i+1) a_{\alpha+e_i}(s) + O(1).\]
This is obtained by writing $\Delta_s\Delta_\tau\varphi = \Delta_\tau\Delta_s\varphi$,
substituting the right-hand side of \eqref{eq:polyfun2} for $\Delta_s\varphi$,
applying $\Delta_\tau$, expanding $(x+\tau)^\alpha$,
and invoking Lemma~\ref{lemma:extractcoeff} to reach a conclusion for the first order Taylor
expansion with respect to $\tau$.

It follows that for each multi-index satisfying $|\beta|=D$,
$a_\beta$ is approximately affine in the sense that
\begin{equation} 
|a_\beta(s)-(u_\beta\cdot s+v_\beta)|=O(1)
\ \text{for nearly all $s\in\rho\ball$}
\end{equation}
for certain 
$u_\beta\in\reals^d\otimes\reals^d$ and $v_\beta\in\reals^d$. 
Insert this conclusion into the hypotheses \eqref{eq:polyfun1},\eqref{eq:polyfun2} to obtain
\[ \Delta_s\varphi(x) = \sum_{|\alpha|=D} (u_\alpha\cdot s+v_\alpha) x^\alpha
+ \sum_{|\alpha|\le D-1} a_\alpha(s) x^\alpha + O(1)\]
for nearly all $(x,s)\in(\rho\ball)^2$.
Once again, there are bothersome terms, $v_\alpha x^\alpha$.
Once again, these can be removed; 
consider $\Delta_s\varphi-\Delta_{s'}\varphi$ and argue as was done for a parallel situation above
to establish \eqref{deltasdeltat}.
One concludes that 
\begin{equation}\label{eq:deltasvarphi} 
\Delta_s\varphi(x) = \sum_{|\alpha|=D} u_\alpha\cdot s x^\alpha
+ \sum_{|\alpha|\le D-1} b_\alpha(s) x^\alpha + O(1)\end{equation}
for nearly all $(x,s)\in(\rho\ball)^2$,
for certain measurable coefficients $b_\alpha$.

We will show below, in Lemma~\ref{lemma:qrequirement}, that
there exists a homogeneous polynomial $q$ of degree $\le D+1$ satisfying
\begin{equation}\label{eq:qrequirement} \Delta_s q(x) \equiv \sum_{|\alpha|=D} u_\alpha\cdot s x^\alpha
+ \sum_{|\alpha|\le D-1} c_\alpha(s)x^\alpha + O(1)\end{equation}
for all $(x,s)\in (\rho\ball)^2$ and for some (polynomial) coefficient functions $c_\alpha$,
with the same $u_\alpha$ as in \eqref{eq:deltasvarphi}.
Granting this for the present, set $\psi = \varphi-q$.
Then
\begin{equation} \Delta_s\psi(x) = \sum_{|\alpha|\le D-1} c_\alpha(s)x^\alpha + O(1)\end{equation}
for nearly all $(x,s)\in(\rho\ball)^2$,
where $c_\alpha$ are measurable functions.
This is the original hypothesis, with $\ball$ replaced by $\rho\ball$,
$\varphi$ replaced by $\psi$, and $D$ replaced by $D-1$.
Therefore it suffices to apply the induction hypothesis to conclude that
$\psi$, and hence $\varphi = \psi+q$, have the required form.
This completes the proof of Theorem~\ref{thm:main}, modulo the proof of the next lemma.
\end{proof}

\begin{lemma} \label{lemma:qrequirement}
There exists a polynomial $q$ of degree $\le D+1$ that satisfies
\eqref{eq:qrequirement}.
\end{lemma}

%Such a polynomial certainly exists
%for $d=1$ (without any $O(1)$ term being needed in the equation);
%a proof of existence for arbitrary dimensions is given below.
%Granting the existence of $q$ for the present, set $\psi = \varphi-q$.  Then

\begin{proof}
Apply $\Delta_t$ to both sides of \eqref{eq:deltasvarphi} to obtain
\begin{align*}
\Delta_t\Delta_s\varphi(x)
&= \Delta_t \sum_{|\alpha|=D}\sum_{j=1}^d u_{\alpha,j}s_j x^\alpha
+ \Delta_t \sum_{|\alpha|\le D-1} b_\alpha(s)x^\alpha + O(1)
\\&= \sum_{|\alpha|=D}\sum_{j=1}^d u_{\alpha,j}s_j \sum_{i=1}^d  \alpha_i x^{\alpha-e_i}t_i
+  \sum_{|\alpha|\le D-2} b_\alpha(s,t)x^\alpha + O(1)
\end{align*}
for nearly all $(x,s,t)\in(\rho\ball)^3$
where $b_\alpha$ are measurable functions.
Since $\Delta_t\Delta_s\varphi = \Delta_s\Delta_t\varphi$, we may
write the corresponding formula for $\Delta_s\Delta_t\varphi$,
equate it to the one derived above, and apply Lemma~\ref{lemma:extractcoeff}
to deduce that 
for each $i,j\in\{1,2,\dots,d\}$,
\begin{equation}
\sum_{|\alpha|=D} u_{\alpha,j}\alpha_ix^{\alpha-e_i}
= \sum_{|\alpha|=D} u_{\alpha,i}\alpha_jx^{\alpha-e_j} + O(1)\end{equation}
for all $x\in\rho\ball$.
Equivalently, for each multi-index $\beta$ satisfying $|\beta|=D-1$,
\begin{equation} \label{eq:ijapproxequality} 
 u_{\beta+e_i,j}(\beta_i+1) = u_{\beta+e_j,i}(\beta_j+1) + O(1)\end{equation}
for each $i,j$.

On the other hand, a homogeneous polynomial $Q$ of degree $D+1$
satisfies the exact relation 
$\Delta_s Q(x) = \sum_{|\alpha|=D} \sum_{j=1}^d \tilde u_{\alpha,j}s_jx^\alpha +R(x,s)$
for some $R$, where $x\mapsto R(x,s)$ is a a polynomial of degree $\le D-1$ for each $s$,
if and only if 
$\partial Q(x)/\partial x_j = \sum_{|\alpha|=D}  \tilde u_{\alpha,j} x^{\alpha}$ 
for each $j\in\{1,2,\dots,d\}$.
This system of equations is solvable for $Q$ if and only if
\begin{equation} \label{eq:ijexactequality}
\sum_{|\alpha|=D} \tilde u_{\alpha,j} \alpha_i x^{\alpha-e_i}
= \sum_{|\alpha|=D} \tilde u_{\alpha,i} \alpha_j x^{\alpha-e_j}
\end{equation}
for all $i\ne j\in\{1,2,\dots,d\}$.
Equivalently, for each multi-index $\beta$ satisfying $|\beta|=D-1$,
\begin{equation} \label{eq:ijequality}
\tilde u_{\beta+e_i,j}(\beta_i+1) = \tilde u_{\beta+e_j,i}(\beta_j+1)\end{equation}
for each $i,j$.

The tuple $(u_{\alpha,k}: |\alpha|=D \text{ and } 1\le k\le d)$
satisfies the system of approximate equations \eqref{eq:ijapproxequality}.
By elementary linear algebra, there exists a tuple $(\tilde u_{\alpha,k})$
with $|\tilde u_{\alpha,k}-u_{\alpha,k}|=O(1)$ for all $\alpha,k$
that satisfies the corresponding system of exact equations \eqref{eq:ijequality}.
This system of equations implies the existence of a homogeneous polynomial $q$
of degree $D+1$ that satisfies
$\partial q(x)/\partial x_j = \sum_{|\alpha|=D}  \tilde u_{\alpha,j} x^{\alpha}$ 
for each $j\in\{1,2,\dots,d\}$.
Therefore
$\Delta_s q(x) = \sum_{|\alpha|=D} \sum_{j=1}^d \tilde u_{\alpha,j}s_jx^\alpha +R(x,s)$
where $R$ is as above.
\end{proof}

The proof of Theorem~\ref{thm:polyfunctleqnmult}
is very similar to that of Theorem~\ref{thm:polyfunctleqn}.
Details are left to the reader. \qed

\section{A final lemma}

The form of the conclusion of the next lemma contrasts with that
of Lemma~\ref{lemma:nearlycharacter}.
In Lemma~\ref{lemma:nearlycharacter}, the logarithms of the factors in the hypothesis
are only nearly determined up to arbitrary additive corrections in $2\pi i \integers$.
In Lemma~\ref{lemma:lastgasp}, no such arbitrary additive corrections arise.

\begin{lemma}\label{lemma:lastgasp}
There exist $A<\infty$ and $\delta>0$ with the following property.
Let $v_j\in\reals$ for $j=1,2$. Let $\eta>0$.
Suppose that
\[ |e^{i(u_1v_1-u_2v_2)}-1|\le\eta\]
for all $(u_1,u_2)\in[0,1]^2$ outside a set of Lebesgue measure $\delta$.
Then $|v_1|+|v_2|\le C\eta$.
\end{lemma}

\begin{proof}
There exists $u_2\in[0,1]$ such that 
\begin{equation} \label{eq:lastgasp} |e^{iu_1v_1}-e^{iu_2v_2}|\le \eta\end{equation}
for all $u_1\in[0,1]$ outside a set $E$ of measure $\le \delta$.
We may assume without loss of generality that $\eta$ is small
and that $v_1\ne 0$. 
Let $A$ be a large constant to be chosen below.
If $|v_1|\ge A\eta$ then 
there must exist an interval $I\subset[0,1]$ of length
comparable to $A|v_1|^{-1}\eta\le 1$ such that $|E\cap I|\le \delta|I|$.
The mapping $I\owns t\mapsto e^{itv_1}$ maps $I$ 
in a measure-preserving manner, up to universal constant factors,
to an arc of the unit circle of length comparable to $A\eta$.
Because $|E\cap I|\le \delta|I|$,
the image of $I\setminus E$ has diameter comparable to $A\eta$.
This contradicts \eqref{eq:lastgasp}. 

Therefore $|v_1|\le A\eta$.
The same reasoning applies to $v_2$.
\end{proof}

\end{document}